\RequirePackage{fix-cm}
\documentclass[smallextended]{svjour3}       
\smartqed  
\usepackage{graphicx}
\usepackage[colorlinks=true,bookmarks=false,linkcolor=blue,urlcolor=blue,citecolor=blue,breaklinks=true]{hyperref}
\usepackage{breakcites}

\usepackage{doi}

\usepackage{amsmath}

\usepackage{mathtools}

\usepackage{amsfonts}
\usepackage{amsthm}
\usepackage{amssymb}

\usepackage{color}
\usepackage{bm} 
\usepackage{dsfont} 

\usepackage{graphicx}

\usepackage{tikz-cd}

\usepackage{tikz}
\definecolor{mycolor1}{rgb}{0.105882,0.619608,0.466667}
\definecolor{mycolor2}{rgb}{0.85098,0.372549,0.00784314}
\definecolor{mycolor3}{rgb}{0.458824,0.439216,0.701961}
\definecolor{mycolor4}{rgb}{0.905882,0.160784,0.541176}
\definecolor{mycolor5}{rgb}{0.4,0.65098,0.117647}
\definecolor{mycolor6}{rgb}{0.65098,0.462745,0.113725}
\definecolor{mycolor7}{rgb}{0.901961,0.670588,0.00784314}
\definecolor{mycolor8}{rgb}{0.4,0.4,0.4}
\definecolor{mycolor9}{rgb}{0.301961,0,0.294118}
\definecolor{mycolor10}{rgb}{0.0313725,0.25098,0.505882}

\usetikzlibrary{calc}        

\makeatletter
\newif\ifmygrid@coordinates
\tikzset{/mygrid/step line/.style={line width=0.80pt,draw=gray!80},
         /mygrid/steplet line/.style={line width=0.25pt,draw=gray!80}}
\pgfkeys{/mygrid/.cd,
         step/.store in=\mygrid@step,
         steplet/.store in=\mygrid@steplet,
         coordinates/.is if=mygrid@coordinates}
\def\mygrid@def@coordinates(#1,#2)(#3,#4){%
    \def\mygrid@xlo{#1}%
    \def\mygrid@xhi{#3}%
    \def\mygrid@ylo{#2}%
    \def\mygrid@yhi{#4}%
}
\newcommand\DrawGrid[3][]{%
    \pgfkeys{/mygrid/.cd,coordinates=true,step=1,steplet=0.2,#1}%
    \draw[/mygrid/steplet line] #2 grid[step=\mygrid@steplet] #3;
    \draw[/mygrid/step line] #2 grid[step=\mygrid@step] #3;
    \mygrid@def@coordinates#2#3%
    \ifmygrid@coordinates%
        \draw[/mygrid/step line]
        \foreach \xpos in {\mygrid@xlo,...,\mygrid@xhi} {%
          (\xpos,\mygrid@ylo) -- ++(0,-3pt)
                              node[anchor=north] {$\xpos$}
        }
        \foreach \ypos in {\mygrid@ylo,...,\mygrid@yhi} {%
          (\mygrid@xlo,\ypos) -- ++(-3pt,0)
                              node[anchor=east] {$\ypos$}
        };
    \fi%
}
\makeatother

\usepackage{psfrag}

\usepackage{algorithm}
\usepackage{algpseudocode}
\algdef{SE}[DOWHILE]{Do}{doWhile}{\algorithmicdo}[1]{\algorithmicwhile\ #1}%

\usepackage{mathrsfs} 

\newcommand{\remove}[1]{}
\newcommand{\removesafe}[1]{}

\newcommand{\argmin}[1]{\underset{#1}{\operatorname{argmin}}}

\newcommand{\transpose}{^\top\! }

\newcommand{\inner}[2]{\left\langle{#1},{#2}\right\rangle}
\newcommand{\innersmall}[2]{\langle{#1},{#2}\rangle}

\newcommand{\trace}{\mathrm{Tr}}
\newcommand{\Trace}{\mathrm{Tr}}

\newcommand{\Proj}{\mathrm{Proj}}

\newcommand{\Exp}{\mathrm{Exp}}

\newcommand{\Retr}{\mathrm{R}}
\newcommand{\inj}{\mathrm{inj}}

\newcommand{\Gr}{\mathrm{Gr}}
\newcommand{\St}{\mathrm{St}}

\newcommand{\T}{\mathrm{T}}

\newcommand{\Rnn}{{\mathbb{R}^{n\times n}}}
\newcommand{\Rnp}{{\mathbb{R}^{n\times p}}}

\newcommand{\Rmn}{{\mathbb{R}^{m\times n}}}

\newcommand{\Rd}{{\mathbb{R}^{d}}}

\newcommand{\Rk}{{\mathbb{R}^{k}}}

\newcommand{\reals}{{\mathbb{R}}}

\newcommand{\Rn}{{\mathbb{R}^n}}

\newcommand{\grad}{\mathrm{grad}}
\newcommand{\Hess}{\mathrm{Hess}}

\newcommand{\diag}{\mathrm{diag}}
\newcommand{\D}{\mathrm{D}}
\newcommand{\dt}{\mathrm{d}t}

\newcommand{\Ddt}{\frac{\D}{\dt}}
\newcommand{\ddt}{\frac{\mathrm{d}}{\mathrm{d}t}}
\newcommand{\ddq}{\frac{\mathrm{d}}{\mathrm{d}q}}

\newcommand{\dtau}{\mathrm{d}\tau}

\newcommand{\calK}{\mathcal{K}}

\newcommand{\calU}{\mathcal{U}}
\newcommand{\calV}{\mathcal{V}}

\newcommand{\calM}{\mathcal{M}}
\newcommand{\calN}{\mathcal{N}}
\newcommand{\calS}{\mathcal{S}}
\newcommand{\calX}{\mathcal{X}}
\newcommand{\calW}{\mathcal{W}}
\newcommand{\Id}{\operatorname{Id}} 

\newcommand{\longg}{{\mathrm{long}}}
\newcommand{\short}{{\mathrm{short}}}

\newcommand{\opnorm}[1]{\left\|{#1}\right\|_\mathrm{op}}

\newcommand{\TODO}[1]{{\color{red}{[#1]}}}
\newcommand{\TODOF}[1]{}

\newcommand{\lambdamin}{\lambda_\mathrm{min}}

\newcommand{\dist}{\mathrm{dist}}

\newcommand{\flow}{f_{\mathrm{low}}}

\newcommand{\defeq}{\triangleq}

\newtheorem{assumption}{A\ignorespaces} 

\usepackage[ansinew]{inputenc}
\usepackage[round,compress]{natbib}
\newcommand{\sigmamin}{\sigma_{\operatorname{min}}}
\newcommand{\sigmamax}{\sigma_{\operatorname{max}}}
\usepackage{subcaption}
\newcommand{\aref}[1]{\hyperref[#1]{A\ref{#1}}}
\newcommand{\tauort}{t}

\begin{document}
	
	\title{Adaptive regularization with cubics on manifolds\thanks{Authors are listed alphabetically. NB was partially supported by NSF award DMS-1719558. CC acknowledges support from The Alan Turing Institute for Data Science, London, UK. NA and BB were supported by Elad Hazan's NSF grant IIS-1523815.}}
	
	
	\author{Naman Agarwal \and Nicolas Boumal \and Brian Bullins \and Coralia Cartis}
	
	\authorrunning{Agarwal, Boumal, Bullins and Cartis} 
	
	\institute{%
		N. Agarwal \at
		Google Research,
		Princeton, NJ \\
		\email{namanagarwal@google.com}
		\and
		N. Boumal \at
		Department of Mathematics,
		Princeton University, NJ \\
		\email{nboumal@math.princeton.edu}
		\and
		B. Bullins \at
		Toyota Technological Institute at Chicago, IL \\
		\email{bbullins@ttic.edu}
		\and
		C. Cartis \at
		Mathematical Institute,
		University of Oxford, UK \\
		\email{coralia.cartis@maths.ox.ac.uk}
	}
	
	\date{Accepted by Mathematical Programming in 2020:\\\url{https://link.springer.com/article/10.1007/s10107-020-01505-1}}
	


\maketitle

\begin{abstract}
Adaptive regularization with cubics (ARC) is an algorithm for unconstrained, non-convex optimization. Akin to the trust-region method, its iterations can be thought of as approximate, safe-guarded Newton steps. For cost functions with Lipschitz continuous Hessian, ARC has optimal iteration complexity, in the sense that it produces an iterate with gradient smaller than $\varepsilon$ in $O(1/\varepsilon^{1.5})$ iterations. For the same price, it can also guarantee a Hessian with smallest eigenvalue larger than $-\sqrt{\varepsilon}$. In this paper, we study a generalization of ARC to optimization on Riemannian manifolds. In particular, we generalize the iteration complexity results to this richer framework. Our central contribution lies in the identification of appropriate manifold-specific assumptions that allow us to secure these complexity guarantees both when using the exponential map and when using a general retraction. A substantial part of the paper is devoted to studying these assumptions---relevant beyond ARC---and providing user-friendly sufficient conditions for them.
Numerical experiments are encouraging. 
\end{abstract}


\section{Introduction} \label{sec:intro}


Adaptive regularization with cubics (ARC) is an iterative algorithm used to solve unconstrained optimization problems of the form
\begin{align*}
	\min_{x \in \Rn} \ f(x),
\end{align*}
where $f \colon \Rn \to \reals$ is twice continuously differentiable~\citep{griewank1981modification}. Given any initial iterate $x_0 \in \Rn$, assuming $f$ is lower-bounded and has a Lipschitz continuous Hessian, ARC produces an iterate $x_k$ with small gradient, namely, $\|\nabla f(x_k)\| \leq \varepsilon$, in at most $O(1/\varepsilon^{1.5})$ iterations~\citep{nesterov2006cubic,cartis2011adaptivecubic,bcgmt}. This improves upon the worst-case iteration complexity of steepest descent and classical trust-region methods.
In fact, this iteration complexity is optimal under those assumptions~\citep{carmon2017lower},
contributing to renewed interest in this method.

In this paper, we study a generalization of ARC to optimization on manifolds, that is, 
\begin{align}
	\min_{x \in \calM} \ f(x),
	\label{eq:P}
	\tag{P}
\end{align}
where $\calM$ is a given Riemannian manifold
and $f \colon \calM \to \reals$ is a (sufficiently smooth) cost function.
The practical interest in optimization on manifolds stems from its ubiquity: it comes up naturally in numerical linear algebra (spectral decompositions, low-rank Lyapunov equations), signal and image processing (shape analysis, diffusion tensor imaging, community detection on graphs, rotational video stabilization), statistics and machine learning (matrix/tensor completion, metric learning, Gaussian mixtures, activity recognition, independent component analysis), robotics and computer vision (simultaneous localization and mapping, structure from motion, pose estimation) and various other fields.
The theoretical interest comes from the fact that Riemannian geometry is arguably the ``right'' setting for \emph{unconstrained} optimization---indeed, it is the minimal mathematical structure required to have comfortable notions of gradients and Hessians, which are the basic building blocks of smooth, unconstrained optimization algorithms.
See for example~\citep{AMS08} and~\citep{boumal2020intromanifolds} for book-length introductions to this topic.
See \emph{related work} below for further references.

Building upon the existing literature for the Euclidean case, we generalize the worst-case iteration complexity analysis of ARC to manifolds, obtaining essentially the same guarantees but with a wider application range: see numerical experiments in Section~\ref{sec:XP} for some examples.

In particular, with the appropriate assumptions discussed in Sections~\ref{sec:firstorderanalysisexp} and~\ref{sec:firstorderanalysisgeneral}, we find that $\varepsilon$-critical points of $f$ on $\calM$ can be computed in $O(1/\varepsilon^{1.5})$ iterations. We also show an iteration complexity bound for the computation of approximate second-order critical points in Section~\ref{sec:secondorder}. Key differences with the Euclidean setting lie in the particular assumptions we make. We further study these assumptions in Sections~\ref{sec:regularity} and~\ref{sec:Dretr}. A subproblem solver---necessary to run ARC---is detailed in Section~\ref{sec:subproblem}. Our algorithm is implemented in the Manopt framework~\citep{manopt} and distributed as part of that toolbox. In Section~\ref{sec:XP}, we close with numerical comparisons to existing solvers, in particular the related Riemannian trust-region method (RTR)~\citep{genrtr}.

\subsection*{Main results}

An important ingredient of ARC on manifolds (Algorithm~\ref{algo:ARC}) is the \emph{retraction} $\Retr$, which allows one to move around the manifold by following tangent vectors. This notion is defined in Section~\ref{sec:arc}. Our results depend on the choice of retraction.

For a twice continuously differentiable cost function $f \colon \calM \to \reals$, the first- and second-order necessary optimality conditions at $x$ read~\citep{yang2012optimality}:
\begin{align*}
	\|\grad f(x)\|_x & = 0, & & \lambdamin(\Hess f(x)) \geq 0,
\end{align*}
where $\grad f$ and $\Hess f$ are the Riemannian gradient and Hessian of $f$---see Section~\ref{sec:firstorderanalysisexp} for definitions; $\|\cdot\|_x$ is the Riemannian norm at $x$ and $\lambdamin$ extracts the smallest eigenvalue of a symmetric operator.

Our first main result applies to \emph{complete} Riemannian manifolds, for which we can use the so-called \emph{exponential map} as retraction $\Retr$. The statement below summarizes more explicit results of Sections~\ref{sec:firstorderanalysisexp} and~\ref{sec:secondorder}, stating iterates of ARC eventually satisfy the necessary optimality conditions up to some tolerance, with a bound on the number of iterations this may require.
\begin{theorem} \label{thm:informalEXP}
	Consider a cost function $f$ on a complete Riemannian manifold $\calM$. If
	\begin{enumerate}
		\item[a)] $f$ is lower bounded (\aref{assu:lowerbound}), and
		\item[b)] the Riemannian Hessian of $f$ is Lipschitz continuous (\aref{assu:firstorderregularityscalar}, \aref{assu:firstorderregularityvector}),
	\end{enumerate}
	then, for any $x_0\in\calM$ and $\varepsilon > 0$, Algorithm~\ref{algo:ARC} with the exponential retraction produces an iterate $x_k \in \calM$ such that $f(x_k) \leq f(x_0)$, $\|\grad f(x_k)\|_{x_k} \leq \varepsilon$ and (if condition~\eqref{eq:secondorderprogress} is enforced) $\lambdamin(\Hess f(x_k)) \geq -\sqrt{\varepsilon}$, with $k = \tilde O(1/\varepsilon^{1.5})$. 
	The bound is dimension- and curvature-free.
\end{theorem}

Our second main result is an extension of the above which allows us to use other retractions, the main motivation being that the exponential map may be unavailable or expensive to compute. We state it as a summary of results in Sections~\ref{sec:firstorderanalysisgeneral} and~\ref{sec:secondorder}.
\begin{theorem} \label{thm:informalRETR}
	Consider a cost function $f$ on a Riemannian manifold $\calM$ equipped with a retraction $\Retr$. If
	\begin{enumerate}
		\item[a)] $f$ is lower bounded (\aref{assu:lowerbound}),
		\item[b)] the pullbacks $f \circ \Retr_x$ satisfy a type of second-order Lipschitz condition (\aref{assu:firstorderregularityscalar}, \aref{assu:firstorderregularityvectorretraction}), and
		\item[c)] the differential of the retraction is well behaved (\aref{assu:DRetr}),
	\end{enumerate}
	then, for any $x_0\in\calM$ and sufficiently small $\varepsilon > 0$, Algorithm~\ref{algo:ARC} with retraction $\Retr$ produces an iterate $x_k \in \calM$ such that $f(x_k) \leq f(x_0)$, $\|\grad f(x_k)\|_{x_k} \leq \varepsilon$ and (if condition~\eqref{eq:secondorderprogress} is enforced and $\Retr$ is second order) $\lambdamin(\Hess f(x_k)) \geq -\sqrt{\varepsilon}$, with $k = \tilde O(1/\varepsilon^{1.5})$. 
\end{theorem}
We further provide sufficient conditions for the assumptions on the pullbacks and the retraction to be satisfied, in Sections~\ref{sec:regularity} and~\ref{sec:Dretr} respectively. For example, \aref{assu:DRetr} is satisfied if the sublevel set of $x_0$ is compact.

\subsection*{Related work}

Numerous algorithms for unconstrained optimization have been generalized to Riemannian manifolds~\citep{luenberger1972gradient,Gabay1982,smith1994optimization,edelman1998geometry,AMS08}, among them gradient descent, nonlinear conjugate gradients, stochastic gradients \citep{bonnabel2013stochastic,zhang2016riemannian}, BFGS~\citep{ring2012optimization}, Newton's method \citep{adler2002spine} and trust-regions~\citep{genrtr}. See these references and also our numerical experiments in Section~\ref{sec:XP} for a discussion of numerous applications.

ARC in particular was extended to manifolds in the PhD thesis of~\citet{qi2011thesis}. There, under a different set of regularity assumptions, asymptotic convergence analyses are proposed, in the same spirit as the analyses presented in the aforementioned references for other methods. Qi also presents local convergence analyses, showing superlinear local convergence under some assumptions.

In contrast, we here favor a global convergence analysis with explicit bounds on iteration complexity to reach approximate criticality.
Such bounds are standard in optimization on Euclidean spaces. Around the same time, they have been generalized to Riemannian gradient descent and other algorithms by \citet{zhang2016complexitygeodesicallyconvex} (focusing on geodesic convexity), by \citet{bento2017iterationcomplexity} (looking also at proximal point methods), and by \citet{boumal2016globalrates} (also analyzing RTR). In the first two works, the regularity assumptions on the cost function are close in spirit to those we lay out in Section~\ref{sec:firstorderanalysisexp}, whereas in the third work the assumptions are closer to our Section~\ref{sec:firstorderanalysisgeneral}.

Closest to our work, \citet{zhang2018cubicregmanifold} recently proposed a convergence analysis of a cubically regularized method on manifolds, also establishing an $O(1/\varepsilon^{1.5})$ iteration complexity. Their analysis (independent from ours: early versions of our results appeared on public repositories around the same time, theirs two weeks before ours) focuses on compact submanifolds of a Euclidean space and uses a fixed regularization parameter (which must be set properly by the user). Subproblems are assumed to be solved to global optimality, though it appears this could be relaxed within their framework. We improve on these points as follows: our analysis is intrinsic (no embedding space is ever referenced), we do not need $\calM$ to be compact, our regularization parameter $\varsigma_k$ is dynamically adapted (which is both easier for the user and more efficient), and the subproblem solver only needs to meet weak requirements to reach approximate criticality. These improvements lead to implementable, competitive algorithms. \citet{zhang2018cubicregmanifold} also study superlinear local convergence rates, in line with~\citet{qi2011thesis} but with different assumptions.

The work by \citet{zhang2018cubicregmanifold} is also related to adaptive \emph{quadratic} regularization on embedded submanifolds of Euclidean space recently studied by~\citet{hu2017adaptive}, where the quadratic model is written in terms of the Euclidean gradient and Hessian.

More recently, two independent papers generalize work by \citet{jin2019escape} to provide iteration complexity bounds for a Riemannian version of perturbed gradient descent, allowing to reach approximate second-order criticality without looking at the Hessian, and with logarithmic dependence in the dimension of the manifold. \citet{sun2019prgd} provide an analysis based on regularity assumptions akin to the ones we lay out in Section~\ref{sec:firstorderanalysisexp}, while \citet{criscitiello2019escapingsaddles} base their analysis on regularity assumptions closer to the ones we lay out in Section~\ref{sec:firstorderanalysisgeneral}.

Our complexity analysis builds on prior work for the Euclidean case by~\citet{cartis2011adaptivecubic} and \citet{bcgmt}. Complexity lower bounds given by \citet{cartis2018worst} and \citet{carmon2017lower} show that the bounds in \citep{nesterov2006cubic,cartis2011adaptivecubic,bcgmt} are optimal in $\varepsilon$-dependency for the appropriate class of functions. A variant of ARC that is closely related to trust-region methods was presented in \citep{dussault2018arcq}.

Recently, various works have focused on efficiently solving the ARC subproblem (that is, minimizing $m_k$ as defined in~\eqref{eq:mk}) in the Euclidean setting. \citet{agarwal2017finding} propose an efficient method to solve the subproblem leading to fast algorithms for converging to second-order local minima in the Euclidean setting. \citet{carmon2016gradient} and \citet{tripuraneni2017stochastic} propose gradient descent--based methods to solve the subproblem.
Several recent papers consider the effect of subsampling on the subproblem~\citep{tripuraneni2017stochastic,kohler2017sub,zhou2018stochastic,zhang2018adaptive,wang2018stochastic}.

In the Riemannian case, the subproblem is posed on a tangent space, which is a linear subspace. Hence, all of the above methods are applicable in the Riemannian setting as well. In particular, we use the Krylov subspace method originally proposed in~\citep{cartis2011adaptivecubic}. Recently, \citet{carmon2018analysis} and~\citet{gould2019regularizedquadratic} provided a bound on the amount of work this method may require to provide sufficient progress (see also Remark~\ref{rem:carmonduchisubproblem}).



On a technical note, in Definition~\ref{def:retrscndnice} we formulate second-order assumptions on the retraction to disentangle the requirements on $f$ from those on the retraction. These are related to (but differ from) the assumptions and discussions in~\citep{ring2012optimization}, specifically Lemma 6, Propositions 5 and 7, and Remarks 2 and 3 in that reference.


\section{ARC on manifolds} \label{sec:arc}

\begin{algorithm}[t]
	\caption{Riemannian adaptive regularization with cubics (ARC)}
	\label{algo:ARC}
	\begin{algorithmic}[1]
		\State \textbf{Parameters: } $\theta > 0$, $\varsigma_{\min} > 0$, $0 < \eta_1 \leq \eta_2 < 1$, $0 < \gamma_1 < 1 < \gamma_2 < \gamma_3$ 
		\State \textbf{Input:} $x_0 \in \calM$, $\varsigma_0 \geq \varsigma_{\min}$
		\For{$k = 0, 1, 2\ldots$}
		\State Consider the pullback $\hat f_k = f \circ \Retr_{x_k} \colon \T_{x_k}\calM \to \reals$. Define the model $m_k$ on $\T_{x_k}\calM$: 
		\begin{align}
			m_k(s) & = \hat f_k(0) + \innersmall{s}{\nabla \hat f_k(0)} + \frac{1}{2}\innersmall{s}{\nabla^2 \hat f_k(0)[s]} + \frac{\varsigma_k}{3} \|s\|^3.
			\label{eq:mk}
		\end{align}
		\State Compute a step $s_k \in \T_{x_k}\calM$ satisfying first-order progress conditions (see Section~\ref{sec:subproblem}):
		\begin{align}
			m_k(s_k) & \leq m_k(0), & \textrm{ and } & & \|\nabla m_k(s_k)\| & \leq \theta \|s_k\|^2.
			\label{eq:firstorderprogress}
		\end{align}
		\Statex \hspace{4.5mm} Optionally, if second-order criticality is targeted, $s_k$ must also satisfy this condition:
		\begin{align}
			\lambdamin(\nabla^2 m_k(s_k)) \geq -\theta \|s_k\|,
			\label{eq:secondorderprogress}
		\end{align}
		\Statex \hspace{4.5mm} where $\lambdamin$ extracts the smallest eigenvalue of a symmetric operator.
		
		\State If $s_k = 0$, terminate (see Lemma~\ref{lem:termination}).
		\State Compute the regularized ratio of actual improvement over model improvement:
		\begin{align}
			\rho_k & = \frac{f(x_k) - f(\Retr_{x_k}(s_k))}{m_k(0) - m_k(s_k) + \frac{\varsigma_k}{3} \|s_k\|^3}.
			\label{eq:rhok}
		\end{align}
		\State If $\rho_k \geq \eta_1$, accept the step: $x_{k+1} = \Retr_{x_k}(s_k)$. Otherwise, reject it: $x_{k+1} = x_k$.
		\State Update the regularization parameter:
		\begin{align}
			\varsigma_{k+1} \in
			\begin{cases}
				\begin{aligned}
					& [\max(\varsigma_{\min}, \gamma_1 \varsigma_k), \varsigma_k] & & \textrm{ if } \rho_k \geq \eta_2 & & \textrm{ (very successful),} \\
					& [\varsigma_k, \gamma_2 \varsigma_k] & & \textrm{ if } \rho_k \in [\eta_1, \eta_2) & & \textrm{ (successful),} \\
					& [\gamma_2 \varsigma_k, \gamma_3 \varsigma_k] & & \textrm{ if } \rho_k < \eta_1 & & \textrm{ (unsuccessful).}
				\end{aligned}
			\end{cases}
			\label{eq:varsigmaupdate}
		\end{align}
		\EndFor
	\end{algorithmic}
\end{algorithm}

ARC on manifolds is listed as Algorithm~\ref{algo:ARC}. It is a direct adaptation from~\citep{cartis2011adaptivecubic,bcgmt}. Like many other optimization algorithms, its generalization to manifolds relies on a chosen {retraction}~\citep{shub1986some,AMS08}. For some $x\in\calM$, let $\T_x\calM$ denote the tangent space at $x$: this is a linear space. Intuitively, a retraction $\Retr$ on a manifold provides a means to move away from $x$ along a tangent direction $s\in\T_x\calM$ while remaining on the manifold, producing $\Retr_x(s) \in \calM$.
For a formal definition, we use the \emph{tangent bundle},
\begin{align*}
	\T\calM & = \{ (x, s) : x \in \calM \textrm{ and } s \in \T_x\calM \},
\end{align*}
which is itself a smooth manifold.
\begin{definition}[{\protect{Retraction~\citep[Def.~4.1.1]{AMS08}}}] \label{def:retraction}
	A \emph{retraction} on a manifold $\calM$ is a smooth mapping $\Retr$ from the tangent bundle $\T\calM$ to $\calM$ with the following properties. Let $\Retr_x \colon \T_x\calM \to \calM$ denote the restriction of $\Retr$ to $\T_x\calM$ through $\Retr_x(s) = \Retr(x, s)$. Then,
	\begin{enumerate}
		\item[(i)] $\Retr_x(0) = x$, where $0$ is the zero vector in $\T_x\calM$; and
		\item[(ii)] The differential of $\Retr_x$ at $0$, $\D\Retr_x(0)$, is the identity map on $\T_x\calM$.
	\end{enumerate}
\end{definition}
In other words: retraction curves $c(t) = \Retr_x(ts)$ are smooth and pass through $c(0) = x$ with velocity $c'(0) = \D\Retr_x(0)[s] = s$.
For the special case where $\calM$ is a linear space, the canonical retraction is $\Retr_x(s) = x+s$. For the unit sphere, a typical retraction is $\Retr_x(s) = \frac{x+s}{\|x+s\|}$.

Importantly, the retraction $\Retr$ chosen to optimize over a particular manifold $\calM$ is part of the algorithm specification. For a given cost function $f$ and a specified retraction $\Retr$, at iterate $x_k$, we define the \emph{pullback} of the cost function to the tangent space $\T_{x_k}\calM$:
\begin{align}
	\hat f_k & = f \circ \Retr_{x_k} \colon \T_{x_k}\calM \to \reals.
	\label{eq:fhatk}
\end{align}
This operation lifts $f$ to a linear space.
We then define a model $m_k \colon \T_{x_k}\calM \to \reals$, obtained as a truncated second-order Taylor expansion of the pullback with cubic regularization: see~\eqref{eq:mk}. We use the notation $\inner{\cdot}{\cdot}_x$ to denote the Riemannian metric on $\T_x\calM$, and we usually simplify this notation to $\inner{\cdot}{\cdot}$ when the base point is clear from context. Likewise, $\|s\|_x = \sqrt{\inner{s}{s}_x}$ is the norm of $s \in \T_x\calM$ induced by the Riemannian metric, and we usually omit the subscript, writing $\|s\|$.
Furthermore, for real functions on linear spaces (such as $\hat f_k$ and $m_k$), we let $\nabla$ and $\nabla^2$ denote the (usual) gradient and Hessian operators.

At iteration $k$, a \emph{subproblem solver} is used to approximately minimize the model $m_k$, producing a \emph{trial step} $s_k$: specific requirements are listed as~\eqref{eq:firstorderprogress} and~\eqref{eq:secondorderprogress}; for the first-order condition, we follow the lead of~\citet{bcgmt}. Section~\ref{sec:subproblem} discusses a practical algorithm.

The quality of the trial step $s_k$ is evaluated by computing $\rho_k$~\eqref{eq:rhok}: the \emph{regularized} ratio of actual to anticipated cost improvement, also following~\citet{bcgmt}. Note that the denominator of $\rho_k$ is equal to the difference between $\hat f_k(0)$ and the second-order Taylor expansion of $\hat f_k$ around 0 evaluated at $s_k$.
If $\rho_k \geq \eta_1$, we accept the trial step and set $x_{k+1} = \Retr_{x_k}(s_k)$: such steps are called \emph{successful}.
Among them, we further identify \emph{very successful} steps, for which $\rho_k \geq \eta_2$; for those, not only is the step accepted, but the regularization parameter $\varsigma_k$ is (usually) decreased.
Otherwise, we reject the step and set $x_{k+1} = x_k$: these steps are \emph{unsuccessful}, and we necessarily increase $\varsigma_k$.

%

We expect Algorithm~\ref{algo:ARC} to produce an infinite sequence of iterates. In practice of course, one would terminate the algorithm as soon as approximate criticality is achieved within some prescribed tolerance. This paper bounds the number of iterations this may require. In the unlikely event that the subproblem solver produces the trivial step $s_k = 0$, the algorithm cannot proceed. Fortunately, this only happens if we reached exact criticality of appropriate order. Proofs are in Appendix~\ref{app:arc}.

%
\begin{lemma}\label{lem:termination}
	If second-order progress~\eqref{eq:secondorderprogress} is not enforced, then the first-order condition~\eqref{eq:firstorderprogress} allows the subproblem solver to return $s_k = 0$ if and only if $\grad f(x_k) = 0$.
	If both~\eqref{eq:firstorderprogress} and~\eqref{eq:secondorderprogress} are enforced, the subproblem solver is allowed to return $s_k = 0$ if and only if $\grad f(x_k) = 0$ and $\Hess f(x_k)$ is positive semidefinite.
\end{lemma}

%
%
%


We now introduce two basic assumptions about the cost function $f$, affording us two supporting lemmas.
The first common assumption is that the cost function $f$ is lower bounded. 
\begin{assumption} \label{assu:lowerbound}
	There exists a finite $\flow$ such that $f(x) \geq \flow$ for all $x \in \calM$.
\end{assumption}
The second assumption is that $f$ is sufficiently differentiable so that the models $m_k$ are well defined, and that second-order Taylor expansions of $\hat f_k$ in the tangent space at $x_k$ are sufficiently accurate. In the Euclidean case, the latter follows from a Lipschitz condition on the Hessian of $f$. In the next two sections, we discuss how this generalizes to manifolds.
\begin{assumption} \label{assu:firstorderregularityscalar}
	The cost function $f$ is twice continuously differentiable.
	Furthermore, there exists a constant $L$ such that, at each iteration $k$, for the trial step $s_k$ selected by the subproblem solver, the pullback $\hat f_k = f \circ \Retr_{x_k}$ satisfies
	\begin{align}
		\hat f_k(s_k) - \left[ \hat f_k(0) + \innersmall{s_k}{\nabla \hat f_k(0)} + \frac{1}{2}\innersmall{s_k}{\nabla^2 \hat f_k(0)[s_k]} \right] & \leq \frac{L}{6} \|s_k\|^3.
	\end{align}
\end{assumption}
The two supporting lemmas below follow the standard Euclidean analysis. The first lemma establishes that the regularization parameter $\varsigma_k$ does not grow unbounded.
\begin{lemma}[{\protect\citet[Lem.~2.2]{bcgmt}}] \label{lem:varsigmamax}
	Under~\aref{assu:firstorderregularityscalar}, the regularization parameter remains bounded: for all $k$, it holds that $\varsigma_k \leq \varsigma_{\max}$, with
	\begin{align}
	\varsigma_{\max} = \max\left( \varsigma_0, \frac{L\gamma_3}{2(1-\eta_2)} \right).
	\label{eq:varsigmamax}
	\end{align}
\end{lemma}
Conditioned on the conclusions of this lemma, the next lemma states that among the first $\bar k$ iterations of ARC, a certain number are sure to be successful.
\begin{lemma}[{\protect\citet[Thm.~2.1]{cartis2011adaptivecubic}}] \label{lem:boundKsuccessfulsteps}
	If $\varsigma_k \leq \varsigma_{\max}$ for all $k$ (as provided by Lemma~\ref{lem:varsigmamax}), then the number $K$ of successful iterations among $0, \ldots, \bar k-1$ satisfies
	\begin{align*}
		\bar k \leq \left(1 + \frac{|\log(\gamma_1)|}{\log(\gamma_2)}\right) K + \frac{1}{\log(\gamma_2)} \log\left(\frac{\varsigma_{\max}}{\varsigma_0}\right).
	\end{align*}
\end{lemma}
In other words, in order to bound the \emph{total} number of iterations ARC may require to attain a certain goal, it is sufficient to bound the number of \emph{successful} iterations that goal may require. The following proposition (extracted from the main proof in~\citep{bcgmt}) further states that this can be done by showing successful steps are not too short.
\begin{proposition} \label{prop:sumnormskcube}
	Let $\{ (x_0, s_0), (x_1, s_1), \ldots \}$ be the set of iterates and trial steps generated by Algorithm~\ref{algo:ARC}. If~\aref{assu:lowerbound} holds, we have
	\begin{align*}
	\sum_{k\in\calS} \|s_k\|^3 \leq \frac{3(f(x_0) - \flow)}{\eta_1 \varsigma_{\min}},
	\end{align*}
	where $\calS$ is the set of successful iterations.
\end{proposition}
\begin{proof}
	By definition, if iteration $k$ is successful, then $\rho_k \geq \eta_1$~\eqref{eq:rhok}. Combining with the first part of the first-order progress condition~\eqref{eq:firstorderprogress} yields
	\begin{align*}
	f(x_k) - f(x_{k+1}) \geq \eta_1\!\left(m_k(0) - m_k(s_k) + \frac{\varsigma_k}{3} \|s_k\|^3 \right)  \geq \frac{\eta_1\varsigma_{\min}}{3} \|s_k\|^3. 
	\end{align*}
	On the other hand, for unsuccessful iterations, $x_{k+1} = x_k$ and the cost does not change.
	Using~\aref{assu:lowerbound}, a telescoping sum yields:
	\begin{align*}
	f(x_0) - \flow & \geq \sum_{k=0}^\infty f(x_k) - f(x_{k+1}) = \sum_{k\in\calS} f(x_k) - f(x_{k+1}) \geq \frac{\eta_1 \varsigma_{\min}}{3} \sum_{k\in\calS} \|s_k\|^3,
	\end{align*}
	as announced.
\end{proof}

\section{First-order analysis with the exponential map}
\label{sec:firstorderanalysisexp}

In this section, we provide a first-order analysis of Algorithm~\ref{algo:ARC} for the case where $\calM$ is a complete manifold and we use the exponential retraction $\Retr = \Exp$---we define these terms momentarily. This notably encompasses the Euclidean case where $\calM = \Rn$, with $\Exp_x(s) = x+s$, as well as all compact or Hadamard manifolds. As such, the results in this section offer a strict generalization of the Euclidean analysis proposed in~\citep{bcgmt} under the assumption of Lipschitz continuous Hessian. An in-depth reference for the Riemannian geometry tools we use is the monograph by~\citet{lee2018riemannian}, while \citet{AMS08} offer an optimization-focused treatment.

On a complete Riemannian manifold $\calM$, for any point $x$ and tangent vector $v \in \T_x\calM$, there exists a unique smooth curve $\gamma_v \colon \reals \to \calM$ such that $\gamma_v(0) = x$, $\gamma_v'(0) = v$ and, for $t < t'$ close enough, $\gamma_v|_{[t, t']}$ is the shortest path connecting $\gamma_v(t)$ to $\gamma_v(t')$. This curve is called a \emph{geodesic}. The \emph{exponential map} is built from these geodesics as the map
\begin{align*}
	\Exp \colon \T\calM \to \calM \colon (x, v) \mapsto \Exp_x(v) = \gamma_v(1).
\end{align*}
This is a smooth map. Because $\Exp_x(tv) = \gamma_{tv}(1) = \gamma_v(t)$, we also find that $\Exp_x(0) = x$ and $\D\Exp_x(0)[v] = v$, so that the exponential map is indeed a retraction (Definition~\ref{def:retraction}).
If the manifold is not complete, then $\Exp$ is only defined on an open subset of $\T\calM$: when we need $\calM$ to be complete, we say so explicitly.

The Riemannian gradient of $f \colon \calM \to \reals$, denoted by $\grad f$, is the vector field on $\calM$ such that $\D f(x)[s] = \inner{\grad f(x)}{s}$, where $\D f(x)[s]$ is the directional derivative of $f$ at $x$ along the tangent direction $s$. One can show that
\begin{align}
	\grad f(x) & = \nabla(f \circ \Exp_x)(0) = \nabla \hat f_x(0),
	\label{eq:gradfpullback}
\end{align}
so that the Riemannian gradient of $f$ at $x$ is nothing but the Euclidean gradient of the pullback $\hat f_x = f \circ \Exp_x$ at the origin of the tangent space $\T_x\calM$ (see also Lemma~\ref{lem:derivativespullback} for a similar statement with retractions).

The Riemannian Hessian of $f$ is the covariant derivative of the gradient vector field, with respect to the Riemannian connection. Denoted by $\Hess f$, it defines a tensor field as follows: $\Hess f(x)$ is a linear operator from $\T_x\calM$ into itself, self-adjoint with respect to the Riemannian metric on that tangent space. Analogously to~\eqref{eq:gradfpullback}, one can show that
\begin{align}
	\Hess f(x) & = \nabla^2 \hat f_x(0),
	\label{eq:HessfHessfhatExp}
\end{align}
which expresses the Riemannian Hessian of $f$ at $x$ as the Euclidean Hessian of the pullback $\hat f_x$ at the origin of $\T_x\calM$. (Here too, see Lemma~\ref{lem:derivativespullback} below.)

These two statements show that the model $m_k$~\eqref{eq:mk} can be written equivalently as
\begin{align}
	m_k(s) & = f(x_k) + \inner{\grad f(x_k)}{s} + \frac{1}{2} \inner{\Hess f(x_k)[s]}{s} + \frac{\varsigma_k}{3} \|s\|^3
	\label{eq:mkExp}
\end{align}
and that~\aref{assu:firstorderregularityscalar} requires
\begin{align}
	f(\Exp_{x_k}(s_k)) - \left[ f(x_k) + \innersmall{s_k}{\grad f(x_k)} + \frac{1}{2}\innersmall{s_k}{\Hess f(x_k)[s_k]} \right] & \leq \frac{L}{6} \|s_k\|^3
	\label{eq:A2storyExp}
\end{align}
for each $(x_k, s_k)$ produced by Algorithm~\ref{algo:ARC}.

In particular, if $\calM$ is a Euclidean space with the exponential map $\Exp_x(s) = x+s$, it is well known that we can secure~\eqref{eq:A2storyExp} if we assume that the Hessian of $f$ is $L$-Lipschitz continuous. This can be written as
\begin{align*}
	\opnorm{\nabla^2 f(x) - \nabla^2 f(y)} & \leq L \|x-y\|,
\end{align*}
where the norm on the left-hand side is the operator norm.
Generalizing this to the Riemannian setting, we face the issue that $\Hess f(x)$ and $\Hess f(y)$ are linear operators defined on distinct tangent spaces (if $x \neq y$): in order to compare them, we need one more tool to compare tangent vectors in distinct tangent spaces.

Given a smooth curve $c \colon [0, 1] \to \calM$ connecting $c(0) = x$ to $c(1) = y$, consider a tangent vector $v \in \T_x\calM$ and a smooth vector field $Z \colon [0, 1] \to \T\calM$ along $c$---that is, $Z(t) \in \T_{c(t)}\calM$---such that $Z(0) = v$.
If the covariant derivative of $Z$ with respect to the Riemannian connection vanishes identically, then we say that $Z$ is a \emph{parallel vector field} along $c$. (For example, the velocity vector field $\gamma'$ of a geodesic $\gamma$ is parallel.) This vector field exists and is unique. We call $Z(1)$ the \emph{parallel transport} of $v$ from $x$ to $y$ along $c$. Parallel transports are linear isometries with respect to the Riemannian metric, and they depend on the chosen path.

Using parallel transports, we can formulate a standard notion of Lipschitz continuity for Riemannian Hessians.
We use the following notation often: given a tangent vector $s \in \T_x\calM$, 
\begin{align}
	P_s \colon \T_x\calM \to \T_{\Exp_x(s)}\calM
	\label{eq:Ps}
\end{align}
denotes parallel transport along the geodesic $\gamma(t) = \Exp_x(ts)$ from $t = 0$ to $t = 1$.
\begin{definition} \label{def:LipschitzHessRiemann}
	A function $f \colon \calM \to \reals$ on a Riemannian manifold $\calM$ has an \emph{$L$-Lipschitz continuous Hessian} if it is twice differentiable and if, for all $(x, s)$ in the domain of $\Exp$,
	\begin{align*}
		\opnorm{ P_s^{-1} \circ \Hess f(\Exp_x(s)) \circ P_s - \Hess f(x)} & \leq L \|s\|. 
	\end{align*}
\end{definition}
For $f$ three times continuously differentiable, this property holds if and only if the covariant derivative of the Riemannian Hessian is uniformly bounded by $L$ (we omit a proof). In particular, this holds with some $L$ for any smooth function on a compact manifold.
In the Euclidean case, parallel transports are identity maps (independent of the transport curve), so that this is equivalent to the usual definition.

Crucially, for cost functions with Lipschitz Hessian, we recover familiar-looking bounds on Taylor expansions of both $f$ itself and, as will be instrumental momentarily, of $\grad f$. Results of this nature are standard: they appear frequently in complexity analyses for Riemannian optimization, see for example~\citep{ferreira2002kantorovichnewton,bento2017iterationcomplexity,sun2019prgd}. Proofs are in Appendix~\ref{app:firstorderexp}.
\begin{proposition} \label{prop:LipschitzHessianBounds}
	Let $f \colon \calM \to \reals$ be twice differentiable on a Riemannian manifold $\calM$. Given $(x, s)$ in the domain of $\Exp$, assume there exists $L \geq 0$ such that, for all $t \in [0, 1]$,
	\begin{align*}
		\opnorm{P_{ts}^{-1} \circ \Hess f(\Exp_x(ts)) \circ P_{ts} - \Hess f(x)} & \leq L \|ts\|. 
	\end{align*}
	Then, the two following inequalities hold:
	\begin{align*}
		\left| f(\Exp_x(s)) - f(x) - \inner{s}{\grad f(x)} - \frac{1}{2} \inner{s}{\Hess f(x)[s]} \right| & \leq \frac{L}{6} \|s\|^3, \textrm{ and }\\
		\left\| P_{s}^{-1} \grad f(\Exp_x(s)) - \grad f(x) - \Hess f(x)[s] \right\| & \leq \frac{L}{2} \|s\|^2.
	\end{align*}
\end{proposition}
This justifies the introduction of the following assumption, still with notation~\eqref{eq:Ps} for $P_{s_k}$.
\begin{assumption} \label{assu:firstorderregularityvector}
	There exists a constant $L'$ such that, at each successful iteration $k$, for the step $s_k$ selected by the subproblem solver, we have
	\begin{align}
		\left\| P_{s_k}^{-1} \grad f(\Exp_{x_k}(s_k)) - \grad f(x_k) - \Hess f(x_k)[s_k] \right\| & \leq \frac{L'}{2} \|s_k\|^2. 
	\end{align}
\end{assumption}
\begin{corollary}
	If $f \colon \calM \to \reals$ has $L$-Lipschitz continuous Hessian, then, for any sequence $\{(x_k, s_k)\}_{k=0,1,2,\ldots}$ in the domain of $\Exp$, $f$ satisfies~\aref{assu:firstorderregularityscalar} with the same $L$ (and $\Retr = \Exp$), and it satisfies~\aref{assu:firstorderregularityvector} with $L' = L$. In particular, if $f$ is smooth and $\calM$ is compact, assumptions \aref{assu:lowerbound}, \aref{assu:firstorderregularityscalar} and \aref{assu:firstorderregularityvector} hold. 
\end{corollary}
We can now state our main result regarding the complexity of running Algorithm~\ref{algo:ARC} on a complete manifold with the exponential retraction, for the purpose of computing approximate first-order critical points. 
We require $\calM$ to be complete so that $\Exp$ is indeed a retraction, defined on the whole tangent bundle.
The proof follows~\citep[Thm.~2.5]{bcgmt} up to the fact that we bound the \emph{total number} of successful iterations that map to points with large gradient (as opposed to bounding the number of such iterations among the first $\bar k$), more in the spirit of~\citep{cartis2012complexity}: this enables us to make a statement about the limit of $\|\grad f(x_k)\|$.
Recall that $\varsigma_{\max}$ is provided by Lemma~\ref{lem:varsigmamax}.
\begin{theorem} \label{thm:masterboundSLipschitzHessian}
	Let $\calM$ be a complete Riemannian manifold and let $\Retr = \Exp$.
	Under~\aref{assu:lowerbound}, \aref{assu:firstorderregularityscalar} and \aref{assu:firstorderregularityvector}, for an arbitrary $x_0 \in \calM$, let $x_0, x_1, x_2\ldots$ be the iterates produced by Algorithm~\ref{algo:ARC}.
	For any $\varepsilon > 0$, the \emph{total number} of \emph{successful} iterations $k$ such that $\|\grad f(x_{k+1})\| > \varepsilon$ is bounded above by
	\begin{align*}
		K_1(\varepsilon) & \triangleq \frac{3(f(x_0) - \flow)}{\eta_1 \varsigma_{\min}} \left(\frac{L'}{2}+\theta+\varsigma_{\max}\right)^{1.5} \frac{1}{\varepsilon^{1.5}}.
	\end{align*}
	Furthermore, $\lim_{k \to \infty} \|\grad f(x_k)\| = 0$ (that is, limit points are critical).
\end{theorem}
\begin{proof}
	If iteration $k$ is successful, we have $x_{k+1} = \Exp_{x_k}(s_k)$.
	The gradient of the model $m_k$~\eqref{eq:mkExp} at $s_k$ (in the tangent space at $x_k$) is given by
	\begin{align*}
	\nabla m_k(s_k) & = \grad f(x_k) + \Hess f(x_k)[s_k] + \varsigma_k \|s_k\| s_k \\
	& = P_{s_k}^{-1}\grad f(x_{k+1}) \\ & \quad + \left( \grad f(x_k) + \Hess f(x_k)[s_k] - P_{s_k}^{-1}\grad f(x_{k+1}) \right) + \varsigma_k \|s_k\| s_k,
	\end{align*}
	with $P_{s_k}$ as defined by~\eqref{eq:Ps}. 
	Owing to the first-order progress condition~\eqref{eq:firstorderprogress}, by the triangle inequality and also using~\aref{assu:firstorderregularityvector}, we find
	\begin{align*}
		\theta \|s_k\|^2 \geq \|\nabla m_k(s_k)\| \geq \|P_{s_k}^{-1}\grad f(x_{k+1})\| - \frac{L'}{2} \|s_k\|^2 - \varsigma_k \|s_k\|^2.
	\end{align*}
	Rearranging and using that $P_{s_k}$ is an isometry, we get for all successful $k$ that
	\begin{align}
		\|\grad f(x_{k+1})\| = \|P_{s_k}^{-1}\grad f(x_{k+1})\| \leq \left(\frac{L'}{2} + \theta + \varsigma_{\max}\right) \|s_k\|^2,
		\label{eq:ineqgradnormxkplusLipschitzHessian}
	\end{align}
	where we also called upon Lemma~\ref{lem:varsigmamax} to claim $\varsigma_k \leq \varsigma_{\max}$.
	Define a subset of the successful steps based on the tolerance $\varepsilon$:
	\begin{align*}
		\calS_\varepsilon & = \{ k : \rho_k \geq \eta_1 \textrm{ and } \|\grad f(x_{k+1})\| > \varepsilon \}.
	\end{align*}
	For $k \in \calS_\varepsilon$, we can lower-bound $\|s_k\|^3$ using~\eqref{eq:ineqgradnormxkplusLipschitzHessian} since $\|\grad f(x_{k+1})\| > \varepsilon$. Then, calling upon Proposition~\ref{prop:sumnormskcube}, we find
	\begin{align*}
		\frac{3(f(x_0) - \flow)}{\eta_1 \varsigma_{\min}} & \geq \sum_{k \in \calS_\varepsilon} \|s_k\|^3 \geq  \frac{\varepsilon^{1.5}}{\left(\frac{L'}{2} + \theta + \varsigma_{\max}\right)^{1.5}} |\calS_\varepsilon|.
	\end{align*}
	This proves the main claim. The claim regarding limit points is proved in Appendix~\ref{app:firstorderexp}.
\end{proof}
(Above, it is natural to consider the sequence $\{x_{k+1}\}$ for successful iterations $k$, as this enumerates each distinct point in the whole sequence once.) 
A key consequence of Theorem~\ref{thm:masterboundSLipschitzHessian} is that, if the number of \emph{successful} iterations among $0, \ldots, \bar k-1$ strictly exceeds $K_1(\varepsilon)$, then it must be that $\|\grad f(x_k)\| \leq \varepsilon$ for some $k$ in $0, \ldots, \bar k$. Combining this with Lemma~\ref{lem:boundKsuccessfulsteps} yields the first main result: a bound on the total number of iterations it may take ARC to produce an approximate critical point on a complete manifold, using the exponential map, and (essentially) assuming a Lipschitz continuous Hessian.
%
%
%
%
%
%
%
%
%
%
%
%
\begin{corollary} \label{cor:mastercorollary}
	Under the assumptions of Theorem~\ref{thm:masterboundSLipschitzHessian}, Algorithm~\ref{algo:ARC} produces a point $x_k \in \calM$ such that $f(x_k) \leq f(x_0)$ and $\|\grad f(x_k)\| \leq \varepsilon$ in at most
	\begin{align*}
		\left(1 + \frac{|\log(\gamma_1)|}{\log(\gamma_2)}\right) K_1(\varepsilon) + \frac{1}{\log(\gamma_2)} \log\left(\frac{\varsigma_{\max}}{\varsigma_0}\right) + 1
	\end{align*}
	iterations.
\end{corollary}
In the Euclidean case, this recovers the result of~\citep{bcgmt} exactly. Note also that this complexity result is unaffected by the curvature of the manifold. Moreover, if $L$ is known and $L' = L$ (which holds under the Lipschitz Hessian assumption), then we can set $\varsigma_0 = \varsigma_{\min} = \frac{1}{2}L$ (so that $\varsigma_{\max} = \frac{\gamma_3}{2(1-\eta_2)}L$) and $\theta = \frac{1}{2}L$. With those choices, we find that
\begin{align}
	K_1(\varepsilon) & = \frac{6}{\eta_1} \left( 1 + \frac{\gamma_3}{2(1-\eta_2)} \right)^{1.5} (f(x_0) - \flow) \sqrt{L} \frac{1}{\varepsilon^{1.5}}.
\end{align}
This exhibits a complexity scaling with $\sqrt{L}$ when $L$ is known, as in~\citep{nesterov2006cubic} and in the lower bound discussed in~\citep{carmon2017lower}.


\section{First-order analysis with a general retraction}
\label{sec:firstorderanalysisgeneral}

The results of the previous section provide a strict, lossless generalization of a known result in the Euclidean case. However, we note two practical shortcomings:
\begin{itemize}
	\item[--] For the analysis to apply, the algorithm must compute the exponential map. 
	\item[--] The Lipschitz condition (Definition~\ref{def:LipschitzHessRiemann}) may be difficult to assess as it involves parallel transports or bounding the covariant derivative of the Riemannian Hessian.
\end{itemize}
Regarding the first point, we organized proofs in Appendix~\ref{app:firstorderexp} to highlight why it is not clear how to generalize Proposition~\ref{prop:LipschitzHessianBounds} to general retractions. In a nutshell, it is because parallel transports and geodesics interact particularly nicely through the fact that the velocity vector field of a geodesic is a parallel vector field.

To address both points, we propose alternate regularity conditions which (a) allow for any retraction, and (b) involve conceptually simpler objects. We do this by focusing on the pullbacks $\hat f_x = f \circ \Retr_x$, which have the merit of being scalar functions on linear spaces---this is in the spirit of prior work~\citep{boumal2016globalrates}. We offer justification for these assumptions below, and in section~\ref{sec:regularity}.

The first regularity assumption, \aref{assu:firstorderregularityscalar}, is readily phrased in terms of the pullback. We focus on providing a replacement for the second condition: \aref{assu:firstorderregularityvector}. Translating this condition to pullbacks by analogy, we aim to bound the difference between $\nabla \hat f_x(s)$---which, conveniently, is a vector tangent at $x$---and a classical truncated Taylor expansion for it: $\nabla \hat f_x(0) + \nabla^2 \hat f_x(0)[s]$. In so doing, it is useful to note that $\nabla \hat f_x(s)$ is related to $\grad f(\Retr_x(s))$ by a linear operator, as follows:
\begin{align}
	\nabla \hat f_x(s) & = T_s^* \grad f(\Retr_x(s)), & \textrm{ with } & & T_s & = \D\Retr_x(s) \colon \T_x\calM \to \T_{\Retr_x(s)}\calM,
	\label{eq:nablahatfTgradf}
\end{align}
where the star indicates the adjoint with respect to the Riemannian metric. Indeed,
\begin{align}
	\forall s, \dot s \in \T_x\calM, \qquad \innersmall{\nabla \hat f_x(s)}{\dot s}_x & = \D \hat f_x(s)[\dot s] \nonumber\\
		& = \D f(\Retr_{x}(s))[ \D\Retr_{x}(s)[\dot s] ] \nonumber\\
		& = \inner{\grad f(\Retr_{x}(s))}{\D\Retr_{x}(s)[\dot s]}_{\Retr_{x}(s)} \nonumber\\
		& = \inner{\left(\D\Retr_{x}(s)\right)^*[\grad f(\Retr_{x}(s))]}{\dot s}_x.
		\label{eq:DRetrnablagrad}
\end{align}
(This also plays a role in~\citep[p599]{ring2012optimization}.)

Considering for a moment how an estimate for $\nabla \hat f_x(s)$ might look like if we use the exponential retraction and under the Lipschitz continuous Hessian assumption~\aref{assu:firstorderregularityvector} as above, we find by triangular inequality that
\begin{align*}
	\|\nabla \hat f_x(s) - \nabla \hat f_x(0) - \nabla^2 \hat f_x(0)[s]\| &
	\leq \|\nabla \hat f_x(s) - P_s^{-1} \grad f(\Exp_x(s))\| \\ & \quad + \|P_s^{-1} \grad f(\Exp_x(s)) - \grad f(x) - \Hess f(x)[s]\| \\
	& \leq \opnorm{T_s^* - P_s^{-1}} \|\grad f(\Exp_x(s))\| + \frac{L'}{2} \|s\|^2,
\end{align*}
using~\eqref{eq:DRetrnablagrad} with $T_s = \D\Exp_x(s)$.
Since parallel transport $P_s$~\eqref{eq:Ps} is an isometry, $P_s^{-1} = P_s^*$ and $\opnorm{T_s^* - P_s^{-1}} = \opnorm{T_s - P_s}$.
For small $s$, we expect $T_s$ (the differential of the exponential map) and $P_s$ (parallel transport) to be nearly the same. Indeed, $\opnorm{T_s - P_s}$ is a continuous function of $s$ and $T_0 = P_0 = \Id$. How much they differ for nonzero $s$ is related to the curvature of the manifold. As a result, we conclude that
\begin{align}
	\left\| \nabla \hat f_x(s) - \nabla \hat f_x(0) - \nabla^2 \hat f_x(0)[s] \right\| & \leq \frac{L'}{2} \|s\|^2 + q(\|s\|) \cdot \|\grad f(\Exp_x(s))\|
	\label{eq:ineqnablahatfsmotivation}
\end{align}
for some continuous function $q \colon \reals^+ \to \reals^+$ such that $q(0) = 0$. 

As an illustration, consider the special case where $\calM$ has constant sectional curvature $C$. In this case, it can be shown using Jacobi fields (see the proof of Lemma~\ref{lemma:exponentialbound} in the appendix) that
\begin{align}
	T_s \dot s & = \D \Exp_x(s)[\dot s] = P_s \dot s + h(\|s\|) P_s\!\left(\dot s - \frac{\inner{s}{\dot s}}{\|s\|^2} s\right),
	\label{eq:DExpConstantCurvature}
\end{align}
where
\begin{align*}
	h(\|s\|) & = \begin{cases}
			0 & \textrm{ if } C = 0, \\
			\frac{\sin(\|s\|/R)}{\|s\|/R} - 1 & \textrm{ if } C = \frac{1}{R^2} > 0, \\
			\frac{\sinh(\|s\|/R)}{\|s\|/R} - 1 & \textrm{ if } C = -\frac{1}{R^2} < 0.
		\end{cases}
\end{align*}
Thus, for manifolds with constant sectional curvature, inequality~\eqref{eq:ineqnablahatfsmotivation} holds with $q(\|s\|) = |h(\|s\|)|$, independent of $x$. This function behaves as $\frac{1}{6}\!\left(\frac{\|s\|}{R}\right)^2 = \frac{C}{6} \|s\|^2$ for small $\|s\|$. This derivation generalizes to manifolds with sectional curvature bounded both from above and from below~\citep{tripuraneni2018averaging}, \citep[Thm.~A.2.9]{waldmann2012geometric}.

Returning to retractions in general, the above motivates us to introduce the following assumption on the pullbacks, meant to replace \aref{assu:firstorderregularityvector}.
\begin{assumption} \label{assu:firstorderregularityvectorretraction}
	There exists a constant $L'$ such that, at each successful iteration $k$, for the step $s_k$ selected by the subproblem solver, the pullback $\hat f_k = f \circ \Retr_{x_k}$ obeys
	\begin{align}
		\left\| \nabla \hat f_k(s_k) - \nabla \hat f_k(0) - \nabla^2 \hat f_k(0)[s_k] \right\| & \leq \frac{L'}{2} \|s_k\|^2 + q(\|s_k\|) \|\grad f(\Retr_{x_k}(s_k))\|,
	\end{align}
	where $q \colon \reals^+ \to \reals^+$ is some continuous function satisfying $q(0) = 0$.
\end{assumption}
Notice how this assumption involves simple tools compared to~\aref{assu:firstorderregularityvector}, which relies on the exponential map and parallel transports. Furthermore, if we strengthen the condition by forcing $q \equiv 0$, we get a Lipschitz-type condition on the pullback: see Section~\ref{sec:regularity}.

Looking at the proof of Theorem~\ref{thm:masterboundSLipschitzHessian}, specifically equation~\eqref{eq:ineqgradnormxkplusLipschitzHessian}, we anticipate the need to lower-bound the norm of $\nabla \hat f_k(s_k)$. Owing to~\eqref{eq:nablahatfTgradf}, it holds that
\begin{align}
	\|\nabla \hat f_k(s)\| \geq \sigmamin\!\left(\D\Retr_{x_k}(s)\right) \|\grad f(\Retr_{x_k}(s))\|,
	\label{eq:nablahatfgradf}
\end{align}
where $\sigmamin$ extracts the smallest singular value of an operator.
For our purpose, it is important that this least singular value remains 
bounded away from zero. This is only a concern for small steps (as large successful steps provide sufficient improvement for other reasons.) Providentially, for $s = 0$, Definition~\ref{def:retraction} ensures $\sigmamin(\D\Retr_{x_k}(0)) = 1$, so that by continuity we expect that it should be possible to meet this requirement. We summarize this discussion in the following assumption.
\begin{assumption} \label{assu:DRetr}
	There exist constants $a > 0$ and $b > 0$ such that, at each successful iteration $k$,
	\begin{align}
		\textrm{if } \|s_k\| \leq a, \textrm{ then } \sigmamin(\D\Retr_{x_k}(s_k)) & \geq b.
	\end{align}
	(The constant $a$ is allowed to be $+\infty$, while $b$ is necessarily at most 1.)
\end{assumption}
In the Euclidean case with $\Retr_x(s) = x+s$, $\D\Retr_x(s)$ is an isometry and one can set $a = +\infty$ and $b = 1$. We secure~\aref{assu:DRetr} in Section~\ref{sec:Dretr} for a large family of manifolds and retractions.

With these new assumptions, we can adapt Theorem~\ref{thm:masterboundSLipschitzHessian} to general retractions. The main change in the proof consists in treating short and long steps separately. This induces a condition that $\varepsilon$ must be small enough for the rate $O(1/\varepsilon^{1.5})$ to materialize. We stress that it is not necessary to know $L, L', q, a$ and $b$ as they appear in \aref{assu:firstorderregularityscalar}, \aref{assu:firstorderregularityvectorretraction} and \aref{assu:DRetr} to run Algorithm~\ref{algo:ARC} in practice: they are only used for the analysis.
Recall that $\varsigma_{\max}$ is provided by Lemma~\ref{lem:varsigmamax}.
\begin{theorem} \label{thm:masterboundSLipschitzHessianretraction}
	Let $\calM$ be a Riemannian manifold equipped with a retraction $\Retr$.
	Under~\aref{assu:lowerbound}, \aref{assu:firstorderregularityscalar}, \aref{assu:firstorderregularityvectorretraction} and \aref{assu:DRetr}, for an arbitrary $x_0 \in \calM$, let $x_0, x_1, x_2\ldots$ be the iterates produced by Algorithm~\ref{algo:ARC}.
	For any $\varepsilon > 0$, the \emph{total number} of \emph{successful} iterations $k$ such that $\|\grad f(x_{k+1})\| > \varepsilon$ is bounded above by
	\begin{align*}
		K_1(\varepsilon) & \triangleq \frac{3(f(x_0) - \flow)}{\eta_1 \varsigma_{\min}} \max\left( \left( \frac{ \frac{L'}{2} + \theta + \varsigma_{\max} }{ b - q(r) } \right)^{1.5} \frac{1}{\varepsilon^{1.5}}, \frac{1}{r^3} \right)
	\end{align*}
	for any $r \in (0, a]$ such that $q(r) < b$.
	Furthermore, $\lim_{k \to \infty} \|\grad f(x_k)\| = 0$.
\end{theorem}
\begin{proof}
	If iteration $k$ is successful, we have $x_{k+1} = \Retr_{x_k}(s_k)$.
	The gradient of the model $m_k$~\eqref{eq:mk} at $s_k$ is given by
	\begin{align*}
		\nabla m_k(s_k) & = \nabla \hat f_k(0) + \nabla^2 \hat f_k(0)[s_k] + \varsigma_k \|s_k\| s_k \\
						& = \nabla \hat f_k(s_k) + \left( \nabla \hat f_k(0) + \nabla^2 \hat f_k(0)[s_k] - \nabla \hat f_k(s_k) \right) + \varsigma_k \|s_k\| s_k.
	\end{align*}
	Owing to the first-order progress condition~\eqref{eq:firstorderprogress}, using~\aref{assu:firstorderregularityvectorretraction} and~\eqref{eq:nablahatfgradf} with $T_{s_k} = \D\Retr_{x_k}(s_k)$,
	\begin{multline*}
		\theta \|s_k\|^2 \geq \|\nabla m_k(s_k)\| \geq \sigmamin(T_{s_k}) \|\grad f(x_{k+1})\| \\ - \frac{L'}{2} \|s_k\|^2 - q(\|s_k\|) \cdot \|\grad f(x_{k+1})\| - \varsigma_k \|s_k\|^2.
	\end{multline*}
	Rearranging and calling upon Lemma~\ref{lem:varsigmamax}, we get for all successful iterations $k$ that
	\begin{align}
		\big( \sigmamin(T_{s_k}) - q(\|s_k\|) \big) \cdot \|\grad f(x_{k+1})\| \leq \left(\frac{L'}{2} + \theta + \varsigma_{\max}\right) \|s_k\|^2.
	\label{eq:ineqgradnormxkplus}
	\end{align}	
	If $k$ is a successful step and $\|s_k\| \leq a$, then~\aref{assu:DRetr} guarantees $\sigmamin(T_{s_k}) \geq b > 0$. Additionally, since $q$ is continuous and satisfies $q(0) = 0$ there necessarily exists $r \in (0, a]$ such that $q(r) < b$. This motivates the following. Recall this subset of the successful steps: 
	\begin{align*}
	\calS_\varepsilon & = \{ k : \rho_k \geq \eta_1 \textrm{ and } \|\grad f(x_{k+1})\| > \varepsilon \}.
	\end{align*}
	Further partition this subset in two, based on step length: \emph{short} steps in $\calS_\short$ and \emph{long} steps in $\calS_\longg$. The partition is based on $r$ as constructed above:
	\begin{align*}
	\calS_\short & = \{ k \in \calS_\varepsilon : \|s_k\| \leq r \}, & \textrm{ and } &  & \calS_\longg & = \calS_\varepsilon \backslash \calS_\short.
	\end{align*}
	
	For $k \in \calS_\short$, we can lower-bound $\|s_k\|^3$ using~\eqref{eq:ineqgradnormxkplus} since $\sigmamin(T_{s_k}) - q(\|s_k\|) \geq b - q(r) > 0$. For $k \in \calS_\longg$, we have $\|s_k\|^3 > r^3$ by definition. Then, calling upon Proposition~\ref{prop:sumnormskcube}, we find
	\begin{align*}
		\frac{3(f(x_0) - \flow)}{\eta_1 \varsigma_{\min}} & \geq \sum_{k \in \calS_\short} \|s_k\|^3 + \sum_{k \in \calS_\longg} \|s_k\|^3 \\ & \geq \frac{ (b - q(r))^{1.5} \varepsilon^{1.5}}{\left(\frac{L'}{2} + \theta + \varsigma_{\max}\right)^{1.5}} |\calS_\short| + r^3 |\calS_\longg| \\
					& \geq \min\left( \left( \frac{ (b - q(r)) \varepsilon }{\frac{L'}{2} + \theta + \varsigma_{\max}} \right)^{1.5}, r^3 \right) |\calS_\varepsilon|.
	\end{align*}
	This proves the main claim. For limit points, see the matching argument in Theorem~\ref{thm:masterboundSLipschitzHessian}.
\end{proof}
A corollary identical to Corollary~\ref{cor:mastercorollary} holds for Theorem~\ref{thm:masterboundSLipschitzHessianretraction} as well.
In the Euclidean case with $\Retr_x(s) = x + s$ and a Lipschitz continuous Hessian, we can set $a = +\infty$, $b = 1$, $q \equiv 0$ and $r = +\infty$, thus also recovering the result of~\citep{bcgmt} exactly.

\TODOF{
\begin{remark}
	\TODO{I tried to find a short-but-correct way to state this, but it didn't quite work out. @Cora, can you suggest something?}
	\TODO{A4 is not needed for *convergence* of ARC (not complexity), even convergence to second order critical points.
		This is actually an important - but not straightforward- comment that we should make in the paper as well. In particular, if we only require say Cauchy decrease for first order criticality (instead of A3) we could prove that grad goes to zero and all limit points are critical. See ARC Part I paper Section 2. It is very similar to Trust region Cauchy-based approaches and so I think that the machinery in the TR complexity on manifold paper would translate here easily.The second-order criticality would stay similar to current ARC approach and it does not require A4.}
\end{remark}
}

\section{Second-order analysis} \label{sec:secondorder}
 
The two previous sections show how to meet first-order necessary optimality conditions approximately. To further satisfy second-order necessary optimality conditions approximately, we also
require second-order progress in the subproblem solver, through condition~\eqref{eq:secondorderprogress}.

This condition is similar to one proposed by~\citet{cartis2017improved} for the same purpose in the Euclidean case.
A direct extension of the proof in that reference would involve the Hessian of the pullback at the trial step $s_k$ rather than at the origin. 
As we have seen for gradients, this leads to technical difficulties.
We provide a proof that achieves the same complexity bound while avoiding such issues.
As a result, there is no need to distinguish between the exponential and the general retraction cases for second-order analysis.

We have the following bound on the total number of successful iterations which can produce points where the Hessian is far from positive semidefinite, akin to Theorems~\ref{thm:masterboundSLipschitzHessian} and~\ref{thm:masterboundSLipschitzHessianretraction}.
\begin{theorem} \label{thm:masterboundSsecond}
	Under~\aref{assu:lowerbound} and~\aref{assu:firstorderregularityscalar}, for an arbitrary $x_0\in\calM$, let $x_0, x_1, x_2\ldots$ be the iterates produced by Algorithm~\ref{algo:ARC} with second-order progress~\eqref{eq:secondorderprogress} enforced. For any $\varepsilon > 0$, the \emph{total number} of \emph{successful} iterations $k$ such that $\lambdamin(\nabla^2 \hat f_k(0)) < -\varepsilon$ is bounded above by
	\begin{align*}
		K_2(\varepsilon) & \triangleq \frac{3(f(x_0) - \flow)}{\eta_1 \varsigma_{\min}} (\theta + 2\varsigma_{\max})^3 \frac{1}{\varepsilon^3}.
	\end{align*}
	Furthermore, $\liminf_{k \to \infty} \lambdamin(\nabla^2 \hat f_k(0)) \geq 0$.
\end{theorem}
\begin{proof}
	The second-order condition~\eqref{eq:secondorderprogress} implies a lower-bound on step-sizes related to the minimal eigenvalue of the Hessian of $\hat f_k = f \circ \Retr_{x_k}$. Indeed, by definition of the model $m_k$~\eqref{eq:mk},
	\begin{align*}
		\forall s, \dot s \in \T_{x_k}\calM, && \nabla^2 m_k(s)[\dot s] & = \nabla^2 \hat f_k(0)[\dot s] + \varsigma_k\left( \|s\| \dot s + \frac{\inner{s}{\dot s}}{\|s\|}s \right).
	\end{align*}
	It follows that
	\begin{align*}
		\lambdamin(\nabla^2 \hat f_k(0)) & = \min_{\|\dot s\| = 1} \inner{\dot s}{\nabla^2 \hat f_k(0)[\dot s]} \\
			& = \min_{\|\dot s\| = 1} \inner{\dot s}{\nabla^2 m_k(s)[\dot s]} - \varsigma_k \left( \|s\|\|\dot s\|^2 + \frac{\inner{s}{\dot s}^2}{\|s\|} \right) \\ & \geq \lambdamin(\nabla^2 m_k(s)) - 2\varsigma_k \|s\|.
	\end{align*}
	In particular, with $s = s_k$, the second-order progress condition~\eqref{eq:secondorderprogress} and Lemma~\ref{lem:varsigmamax} yield
	\begin{align}
		-\lambdamin(\nabla^2 \hat f_k(0)) & \leq (\theta + 2\varsigma_{\max}) \|s_k\|.
		\label{eq:secondorderstepsizebig}
	\end{align}
	Consider this particular subset of the successful iterations:
	\begin{align*}
		\calS_\varepsilon & = \{  k : \rho_k \geq \eta_1 \textrm{ and } \lambdamin(\nabla^2 \hat f_k(0)) < -\varepsilon \}.
	\end{align*}
	Using Proposition~\ref{prop:sumnormskcube} with~\eqref{eq:secondorderstepsizebig} on this set leads to:
	\begin{align*}
		\frac{3(f(x_0) - \flow)}{\eta_1 \varsigma_{\min}} \geq |\calS_\varepsilon| \left(\frac{\varepsilon}{\theta + 2\varsigma_{\max}}\right)^3,
	\end{align*}
	which is the desired bound on the number of steps in $\calS_\varepsilon$. The limit inferior result follows from an argument similar to that at the end of the proof of Theorem~\ref{thm:masterboundSLipschitzHessian}.
\end{proof}
Here too, if $L$ is known we can set $\varsigma_0 = \varsigma_{\min} = \frac{1}{2}L$ (so that $\varsigma_{\max} = \frac{\gamma_3}{2(1-\eta_2)}L$) and $\theta = \frac{1}{2}L$. With those choices, we find that for a target depending on $L$ we have:
\begin{align}
	K_2(\sqrt{L\varepsilon}) & = \frac{6}{\eta_1} \left(\frac{1}{2} + \frac{\gamma_3}{(1-\eta_2)}\right)^3 (f(x_0) - \flow) \sqrt{L} \frac{1}{\varepsilon^{1.5}}.
\end{align}
This exhibits a complexity scaling with $\sqrt{L}$ when $L$ is known.

Theorem~\ref{thm:masterboundSsecond} is a statement about the Hessian of the pullbacks, $\nabla^2 \hat f_k(0)$, whereas we would more naturally desire a statement about the Hessian of the cost function itself, $\Hess f(x_k)$. For the exponential retraction, these two objects are the same~\eqref{eq:HessfHessfhatExp}. More generally, they are the same for any \emph{second-order retraction} (of which the exponential map is one example)~\citep[\S5]{AMS08}: we defer their (standard) definition to Section~\ref{sec:regularity}. For now, accepting the claim that for second-order retractions we have $\Hess f(x_k) = \nabla^2 \hat f_k(0)$, we get a more directly useful corollary: a complexity result for the computation of approximate second-order critical points on manifolds. The proof is in Appendix~\ref{app:secondorder}.
\begin{corollary} \label{cor:mastercorollarysecond}
	Under~\aref{assu:lowerbound} and~\aref{assu:firstorderregularityscalar},
	for an arbitrary $x_0\in\calM$
	and for any $\varepsilon_{g}, \varepsilon_{H} > 0$,
	if either
	\begin{enumerate}
		\item[(a)] we use the exponential retraction and~\aref{assu:firstorderregularityvector} holds, or
		\item[(b)] we use a second-order retraction and both~\aref{assu:firstorderregularityvectorretraction} and~\aref{assu:DRetr} hold,
	\end{enumerate}
	then
	Algorithm~\ref{algo:ARC} with second-order progress~\eqref{eq:secondorderprogress} enforced
	produces a point $x_k \in \calM$ such that $f(x_k) \leq f(x_0)$, $\|\grad f(x_k)\| \leq \varepsilon_{g}$ and $\lambdamin(\Hess f(x_k)) \geq -\varepsilon_{H}$
	in at most
	\begin{align*}
		\left(1 + \frac{|\log(\gamma_1)|}{\log(\gamma_2)}\right)\left({K_1(\varepsilon_{g}) + K_2(\varepsilon_{H}) + 1}\right) + \frac{1}{\log(\gamma_2)}\log\left(\frac{\varsigma_{\max}}{\varsigma_0}\right) + 1
	\end{align*}
	iterations, with $K_1$ as provided by Theorem~\ref{thm:masterboundSLipschitzHessian} or~\ref{thm:masterboundSLipschitzHessianretraction}, depending on assumptions.
\end{corollary}
If the retraction is not second order, we still get a bound on the eigenvalues of the Riemannian Hessian if the retraction has bounded acceleration at the origin at $x_k$: see~\citep[\S3.5]{boumal2016globalrates} and also Lemma~\ref{lem:derivativespullback} below.


\section{Regularity assumptions} \label{sec:regularity}

The regularity assumptions~\aref{assu:firstorderregularityscalar} and~\aref{assu:firstorderregularityvectorretraction} pertain to the pullbacks $f \circ \Retr_{x_k}$. As such, they mix the roles of $f$ and $\Retr$. The purpose of this section is to shed some light on these assumptions, specifically in a way that disentangles the roles of $f$ and $\Retr$. In that respect, the main result is Theorem~\ref{thm:regularitycompact} below. Proofs are in Appendix~\ref{app:regularity}.

Each pullback is a function from a Euclidean space $\T_{x_k}\calM$ to $\reals$, so that standard calculus applies. Since the retraction is smooth by definition, pullbacks are as many times differentiable as $f$. This leads to the following simple fact.
\begin{lemma} \label{lem:lipschitzhessianbasic}
	Assume $f \colon \calM \to \reals$ is twice continuously differentiable.
	If there exists $L$ such that, for all $(x, s)$ among the sequence of iterates and trial steps $\{(x_0, s_0), (x_1, s_1), \ldots\}$ produced by Algorithm~\ref{algo:ARC}, with $\hat f = f \circ \Retr_x$, it holds that
	\begin{align}
		\left\| \nabla^2 \hat f(ts) - \nabla^2 \hat f(0) \right\|_{\mathrm{op}} & \leq t L \|s\|
		\label{eq:lipschitzhessianbasic}
	\end{align}
	for all $t \in [0, 1]$, then~\aref{assu:firstorderregularityscalar} holds with this $L$ and~\aref{assu:firstorderregularityvectorretraction} holds with $L' = L$ and $q \equiv 0$.
\end{lemma}

We call this a \emph{Lipschitz-type} assumption on $\nabla^2 \hat f$ because it compares the Hessians at~$ts$ and~0, rather than comparing them at two arbitrary points on the tangent space. On the other hand, we require this to hold on several tangent spaces with the same constant $L$.

In light of Lemma~\ref{lem:lipschitzhessianbasic}, one way to understand our regularity assumptions is to understand the Hessian of the pullback at points which are not the origin. The following lemma provides the necessary identities. The Hessian formula we have not seen elsewhere.
%
Notation-wise, recall that $T^*$ denotes the adjoint of a linear operator $T$; furthermore, the \emph{intrinsic acceleration} $c''(t)$ of a smooth curve $c(t)$ is the covariant derivative of its velocity vector field $c'(t)$ on the Riemannian manifold $\calM$.
\begin{lemma} \label{lem:derivativespullback}
	Given $f \colon \calM \to \reals$ twice continuously differentiable and $x \in \calM$, the gradient and Hessian of the pullback $\hat f = f \circ \Retr_x$ at $s \in \T_x\calM$ are given by
	\begin{align}
		\nabla \hat f(s) & = T_s^* \grad f(\Retr_x(s)), \label{eq:gradientpullback} \\
		\nabla^2 \hat f(s) & = T_s^* \circ \Hess f(\Retr_x(s)) \circ T_s + W_s,
		\label{eq:hessianpullback}
	\end{align}
	where
	\begin{align}
		T_s & = \D\Retr_x(s) \colon \T_{x}\calM \to \T_{\Retr_x(s)}\calM
		\label{eq:Ts}
	\end{align}
	is linear, and $W_s$ is a symmetric linear operator on $\T_x\calM$ defined through polarization by
	\begin{align}
		\inner{W_s[\dot s]}{\dot s} & = \inner{\grad f(\Retr_x(s))}{c''(0)},
		\label{eq:Whessianpullback}
	\end{align}
	with $c''(0) \in \T_{\Retr_x(s)}\calM$ the intrinsic acceleration on $\calM$ of $c(t) = \Retr_x(s+t\dot s)$ at $t = 0$.
\end{lemma}
The particular case $s = 0$ connects to the comments around Corollary~\ref{cor:mastercorollarysecond}: since $T_0$ is identity by definition of retractions, we find that
\begin{align}
	\nabla^2 \hat f(0) & = \Hess f(x) + W_0,
	\label{eq:nablatwohatfHessfWzero}
\end{align}
where $W_0$ is zero in particular if the initial acceleration of retraction curves is zero (or if $\grad f(x) = 0$). As in~\citep[\S5]{AMS08}, this motivates the definition of second-order retractions, for which $\nabla^2 \hat f(0) = \Hess f(x)$.
\begin{definition}[Second-order retraction]\label{def:retrsecond}
	A retraction $\Retr$ on $\calM$ is \emph{second order} if, for any $x\in\calM$ and $\dot s \in \T_x\calM$, the curve $c(t) = \Retr_x(t\dot s)$ has zero initial acceleration: $c''(0) = 0$.
\end{definition}
Practical second-order retractions are often available~\citep[Ex.~23]{absil2012retractions}.
To prove our main result, we further restrict retractions.
\TODOF{Make a more explicit comment regarding manifolds with negative curvature, since now have full details stating the operator norm of $T_s$ is $\cosh(\|s\|)$ on hyperbolic space, and that goes to infinity (exponentially) with $\|s\|$.---On the other hand, I'm not going to do much about it here, and it's already long. At least, as written now, the door is open for further refinements.}
\begin{definition}[Second-order nice retraction] \label{def:retrscndnice}
	Let $S$ be a subset of the tangent bundle $\T\calM$.
	A retraction $\Retr$ on $\calM$ is \emph{second-order nice on $S$} if there exist constants $c_1, c_2, c_3$ such that, for all $(x, s) \in S$ and for all $\dot s \in \T_x\calM$, all of the following hold:
	\TODOF{Are exponential maps second-order nice? Certainly not on manifolds with sectional curvature upper bounded by a negative number, because then even the sigmamin of DRetr is lower bounded by $\sinh(\|s\|)/\|s\|$ (roughly; check the Jacobi result), and this goes to infinity with $\|s\|$ (and pretty fast, too!) We may have to put in some restrictions on norm of $\|s\|$ (it could work because around 0 that function is pretty flat; though remember: it's the sigmamin, not the sigmamax) ... See notes end of NB36, Nov.\ 21, 2018}
	\begin{enumerate}
		\item $\forall \tauort \in [0, 1]$, $\|T_{ts}\|_{\mathrm{op}} \leq c_1$ where $T_{ts} = \D\Retr_x(ts)$; 
		\item $\forall \tauort \in [0, 1]$, the covariant derivative of $U(\tauort) = T_{\tauort s} \dot s$ satisfies  $\left\| \frac{\D}{\mathrm{d}\tauort} U(\tauort) \right\| \leq c_2 \|s\|\|\dot s\|$;  and
		\item $\|c''(0)\| \leq c_3 \|s\| \|\dot s\|^2 \textrm{ where } c(t) = \Retr_x(s+t\dot s)$.
	\end{enumerate}
	If this holds for $S = \T\calM$, we say $\Retr$ is \emph{second-order nice}.
\end{definition}
Second-order nice retractions are, in particular, second-order retractions (consider $s = 0$ in the last condition).
For $\calM$ a Euclidean space, the canonical retraction $\Retr_x(s) = x+s$ is second-order nice with $c_1 = 1$ and $c_2 = c_3 = 0$.
The classical retraction on the sphere is also second-order nice, with small constants $c_1, c_2, c_3$. We expect this to be the case for many usual retractions on compact or flat manifolds. 
For manifolds with negative curvature, the exponential retraction would lead $\opnorm{T_{s}}$ to grow arbitrarily large with $s$ going to infinity, hence it is important to consider restrictions to appropriate subsets, or to use another retraction.
\begin{proposition}\label{prop:sphereretrnice}
	For the unit sphere $\calM = \{ x\in\Rn : \|x\| = 1 \}$ as a Riemannian submanifold of $\Rn$, the retraction $\Retr_x(s) = \frac{x+s}{\|x+s\|}$ is second-order nice with $c_1 = 1, c_2 = \frac{4\sqrt{3}}{9}, c_3 = 2$.
\end{proposition}
\begin{theorem} \label{thm:regularitycompact}
	Let $f \colon \calM \to \reals$ be three times continuously differentiable.
	Assume the retraction is second-order nice on the set $\{(x_0, s_0), (x_1, s_1), \ldots \}$ of points and steps  generated by Algorithm~\ref{algo:ARC} (see Definition~\ref{def:retrscndnice}).
	If the sequence $x_0, x_1, x_2, \ldots$ remains in a compact subset of $\calM$,
	then \aref{assu:firstorderregularityscalar} and \aref{assu:firstorderregularityvectorretraction} are satisfied
	with a same $L$ (related to the Lipschitz properties of $f$, $\grad f$ and $\Hess f$: see the proof for an explicit expression) and $q \equiv 0$.
	%
	
	Algorithm~\ref{algo:ARC} is a descent method, so that the last condition holds in particular if the sublevel set $\{x \in \calM : f(x) \leq f(x_0)\}$ is compact, and a fortiori if $\calM$ is compact.
\end{theorem}
%

This general result shows existence of (loose) bounds for the Lipschitz constants. For specific optimization problems, it is sometimes easy to derive more accurate constants by direct computation: see for example~\citep[App.~D]{criscitiello2019escapingsaddles} for PCA.


\section{Controlling the differentiated retraction} \label{sec:Dretr}

In Theorem~\ref{thm:masterboundSLipschitzHessianretraction}, we control the worst-case running time of ARC via the differential of the retraction $\Retr$, through~\aref{assu:DRetr}. This assumption, which does not come up in the Euclidean case, involves constants $a, b$ to control $\sigmamin(\D\Retr_{x_k}(s_k))$. Our first result shows~\aref{assu:DRetr} is satisfied for some $a$ and $b$ for any retraction on any manifold, provided the sublevel set of $f(x_0)$ is compact (see also Theorem~\ref{thm:regularitycompact}). This is mostly a topological argument. Proofs are in Appendix~\ref{sec:proofsDretr}.
\begin{theorem} \label{thm:retractionsgeneralmainnew}
	Let $\Retr$ be a retraction on a Riemannian manifold $\calM$, and let $\calU$ be a nonempty compact subset of $\calM$. For any $b \in (0, 1)$ there exists $a > 0$ such that, for all $x \in \calU$ and $s \in \T_x\calM$ with $\|s\|_x \leq a$, we have $\sigmamin(\D\Retr_x(s)) \geq b$. In particular, \aref{assu:DRetr} is satisfied with such $(a, b)$ provided the iterates $x_0, x_1, x_2, \ldots$ remain in $\calU$.
\end{theorem}

We further quantify the constants $a$ and $b$ in two cases of interest:
\begin{enumerate}
	\item For the Stiefel manifold $\St(n,p) = \{X \in \mathbb{R}^{n \times p} : X\transpose X = I_p\}$ as a Riemannian submanifold of $\Rnp$ with the usual inner product $\inner{A}{B} = \trace(A\transpose B)$, we explicitly control $(a, b)$ for the popular Q-factor retraction ($\Retr_X(S)$ is obtained by Gram--Schmidt orthonormalization of the columns of $X+S$). Special cases include the sphere ($p=1$) and the orthogonal group ($p = n$).
	\item For complete manifolds with bounded sectional curvature, we control $(a, b)$ for the case of the exponential retraction. Important special cases include Euclidean spaces (flat manifolds), manifolds with nonpositive curvature (Hadamard manifolds, including the manifold of positive definite matrices~\citep{moakher2006symmetric,Bhatia}), and compact manifolds \citep[\S9.3, p166]{bishop1964geometry}.
\end{enumerate}

The first result follows from a direct calculation. 
\begin{proposition} \label{prop:stiefelmastercorollary}
	For the Stiefel manifold with the Q-factor retraction,
	for any $a > 0$, define $b  = 1 - 3a - \frac{1}{2}a^2$. If $b$ is positive, then~\aref{assu:DRetr} holds with these $a$ and $b$. Moreover, for the sphere, we have that~\aref{assu:DRetr} is satisfied for any $a > 0$ and $b = \frac{1}{1 + a^2}$.
\end{proposition}

The second result follows from the connection between the differential of the exponential map and certain \emph{Jacobi fields} on $\calM$, together with standard comparison theorems from Riemannian geometry~\citep[Ch.~10, 11]{lee2018riemannian}.
\begin{proposition} \label{prop:mastercorollaryexponetial}
	Let $\calM$, complete,
	have sectional curvature upper bounded by $C$, and let the retraction $\Retr$ be the exponential retraction $\Exp$:
	\begin{itemize}
		\item[] If $C \leq 0$, then~\aref{assu:DRetr} is satisfied for any $a > 0$ and $b = 1$;
		\item[] If $C > 0$, then~\aref{assu:DRetr} is satisfied for any $0 < a < \frac{\pi}{\sqrt{C}}$ and $b = \frac{\sin(a\sqrt{C})}{a\sqrt{C}}$.
	\end{itemize}
\end{proposition}


\section{Solving the subproblem} \label{sec:subproblem}


\TODOF{Parlett, 'A New Look at the Lanczos Algorithm for Solving Symmetric Systems of Linear Equations' has 'old tricks' to do Lanczos right in the face of round-off error; it's rather involved though; it'd be nice to get into that 'at some point'.}

At each iteration, Algorithm~\ref{algo:ARC} requires the approximate minimization of the model $m_k$~\eqref{eq:mk} in the tangent space $\T_{x_k}\calM$. Since the latter is a linear space, this is the same subproblem as in the Euclidean case. In contrast to working simply over $\Rn$ however, one practical difference is that we do not usually have access to a preferred basis for $\T_{x_k}\calM$, so that it is preferable to resort to basis-free solvers. To this end, we describe a Lanczos method as Algorithm~\ref{algo:lanczosoptimized}.

Let us phrase the subproblem in a general context. Given a vector space $\calX$ of dimension $n$ with an inner product $\inner{\cdot}{\cdot}$ (and associated norm $\|\cdot\|$), an element $g \in \calX$, a self-adjoint linear operator $H \colon \calX \rightarrow \calX$ and a real $\varsigma > 0$, define the function $m\colon\mathcal{X} \rightarrow \reals$ as 
\begin{align}
	m(s) = \inner{g}{s} + \frac{1}{2}\inner{s}{H(s)} + \frac{\varsigma}{3}\|s\|^3. 	
	\label{eqn:modelfunctionm}
\end{align}
We wish to compute an element $s \in \calX$ such that
\begin{align}
	m(s) & \leq m(0) & \textrm{ and } & & \|\nabla m(s)\| & \leq \theta \|s\|^2.
	\label{eqn:subsolvercondition}
\end{align}
This corresponds to satisfying condition~\eqref{eq:firstorderprogress}
at iteration $k$, where $\calX = \T_{x_k}\calM$ is endowed with the Riemannian inner product at $x_k$, $g = \nabla \hat f_k(0)$, $H = \nabla^2 \hat f_k(0)$ and $\varsigma = \varsigma_k$. If $g = 0$, then $s = 0$ satisfies the condition: henceforth, we assume $g \neq 0$.

Certainly, a global minimizer of~\eqref{eqn:modelfunctionm} meets our requirements (it would also satisfy the equivalent of the second-order condition~\eqref{eq:secondorderprogress}). Such a minimizer can be computed, but known procedures for this task involve a diagonalization of $H$, which may be expensive. Instead, we use the Lanczos-based method proposed in~\citep[\S6]{cartis2011adaptivecubicPartI}: 
the latter iteratively produces a sequence of orthonormal vectors $\{q_1, \ldots, q_n\}$ and a symmetric tridiagonal matrix $T$ of size $n$ such that~\citep[Lec.~36]{trefethen1997numerical}\footnote{In case of so-called breakdown in the Lanczos iteration at step $k$, we follow the standard procedure which is to generate $q_k$ as a random unit vector orthogonal to $q_1, \ldots, q_{k-1}$, then to proceed as normal. This does not jeopardize the desired properties~\eqref{eqn:lanczosmethod}.\label{footnotelanczos}}
\begin{align}
	q_1 & = \frac{g}{\|g\|} &  \textrm{ and } & & T_{ij} & = \inner{q_i}{H(q_j)} \textrm{ for all } i,j \textrm{ in } 1\ldots n.
	\label{eqn:lanczosmethod}
\end{align}
Let $T_k$ denote the $k \times k$ principal submatrix of $T$: producing $q_1, \ldots, q_k$ and $T_k$ requires exactly $k$ calls to $H$.
%
%
Consider $m(s)$~\eqref{eqn:modelfunctionm} restricted to the subspace spanned by $q_1, \ldots, q_k$:
\begin{align}
	\forall y \in \Rk, \textrm{ with } s & = \sum_{i = 1}^{k} y_i q_i, &   m(s) & = \inner{g}{y_1q_1} + \frac{1}{2} \sum_{i,j=1}^k y_i y_j \inner{q_i}{H(q_j)} + \frac{\varsigma}{3} \|y\|^3 \nonumber\\
		 & & & = y_1 \|g\| + \frac{1}{2} y\transpose T_ky + \frac{\varsigma}{3} \|y\|^3,
		 \label{eq:modelrestrictedlanczos}
\end{align}
where $\|\cdot\|$ is also the 2-norm over $\Rk$.
Since $T_k$ is tridiagonal, it can be diagonalized efficiently. As a result, it is inexpensive to compute a global minimizer of $m(s)$ restricted to the subspace spanned by $q_1, \ldots, q_k$~\citep[\S6.1]{cartis2011adaptivecubicPartI}.
%
Furthermore, since the Lanczos basis is constructed incrementally, we can minimize the restricted cubic at $k = 1$, check the stopping criterion~\eqref{eqn:subsolvercondition}, and proceed to $k = 2$ only if necessary, etc. The hope (borne out in experiments) is that the algorithm stops well before $k$ reaches $n$ (at which point it necessarily succeeds.) In this way, we limit the number of calls to $H$, which is typically the most expensive part of the process. This general strategy was first proposed in the context of trust-region subproblems by~\citet{gould1999solving}. Algorithm~\ref{algo:lanczosoptimized} is set up to target first order progress only. In order to ensure satisfaction of second-order progress, one may have to force the execution of $n$ iterations (which is rarely done in practice).





We draw attention to a technical point. Upon minimizing~\eqref{eq:modelrestrictedlanczos}, we obtain a vector $y \in \Rk$. To check the stopping criterion~\eqref{eqn:subsolvercondition}, we must compute $\|\nabla m(s)\|$, where $s = \sum_{i=1}^k y_i q_i$. Since
\begin{align}
	\nabla m(s) & = g + H(s) + \varsigma \|s\| s,
\end{align}
one approach involves computing $s$ (that is, form the linear combination of $q_i$'s) and applying $H$ to~$s$: both operations may be expensive in high dimension. An alternative (shown in Algorithm~\ref{algo:lanczosoptimized}) is to recognize that, due to the inner workings of Lanczos iterations, $\nabla m(s)$ lies in the subspace spanned by $\{q_1, \ldots, q_{k+1}\}$ (if $k < n$). Explicitly, 
\begin{align}
	\nabla m(s) & = \|g\|q_1 + \sum_{i = 1}^{k+1} (T_{1:k+1,1:k}^{} \, y)_i q_i + \varsigma\|y\| \sum_{i=1}^{k}y_iq_i,
	\label{eqn:subsolvernormofgrad}
\end{align}
where $T_{1:k+1,1:k}$ is the submatrix of $T$ containing the first $k+1$ rows and first $k$ columns.
This expression gives a direct way to compute $\|\nabla m(s)\|$ without forming $s$ and without calling $H$,
simply by running the Lanczos iteration one step ahead.


\begin{remark} \label{rem:carmonduchisubproblem}
	\citet{carmon2018analysis} analyze the number of Lanczos iterations that may be required to reach approximate solutions to the subproblem. For the Euclidean case, they conclude that the overall complexity of ARC to compute an $\varepsilon$-critical point in terms of Hessian-vector products (which dominate the number of cost and gradient computations) is $O(\varepsilon^{-7/4})$. Furthermore, the dependence on the dimension of the search space is only logarithmic. (The logarithmic terms are caused by the need to randomize for the so-called hard case---see the reference for important details in that regard.)
	Their conclusions should extend to the Riemannian setting as well.
%
	\citet{gould2019regularizedquadratic} extend these results to study the decrease of the norm of the gradient specifically, as required here.
\end{remark}

\begin{algorithm}[t]
	\caption{Lanczos-based cubic model subsolver}
	\label{algo:lanczosoptimized}
	\begin{algorithmic}[1]
		\State \textbf{Parameters:} $\theta , \varsigma > 0$, vector space $\calX$ with inner product $\inner{\cdot}{\cdot}$ and norm $\|\cdot\|$
		\State \textbf{Input:} $g \in \calX$ nonzero, a self-adjoint linear operator $H\colon\calX \rightarrow \calX$
		\Statex 
		\State{$k \leftarrow 1$}
		\State{Obtain $q_1, T_1$ via a Lanczos iteration \eqref{eqn:lanczosmethod}}
		\State{Solve $y^{(1)} = \argmin{y \in \reals}{\;\;\|g\|y + \frac{1}{2}T_{1}y^2 + \frac{1}{3}\varsigma |y|^3}$ }
		\State{Obtain $q_2, T_2$ via a Lanczos iteration \eqref{eqn:lanczosmethod}}
		\State{Compute $\|\nabla m(s_1)\|$ via~\eqref{eqn:subsolvernormofgrad}}, where $s_1 = y^{(1)} q_1$
		\Statex
		\While{$\|\nabla m(s_k)\| > \theta \|s_k\|^2$}
		\State{$k \leftarrow k + 1$}
		\State{Solve $y^{(k)} = \argmin{y \in \reals^k}{\;\;\|g\| y_1 + \frac{1}{2}y\transpose T_ky + \frac{1}{3}\varsigma \|y\|^3}$ following \citep[\S6.1]{cartis2011adaptivecubicPartI}}
		\State{Obtain $q_{k+1}, T_{k+1}$ via a Lanczos iteration \eqref{eqn:lanczosmethod}}
		\State{Compute $\|\nabla m(s_k)\|$ via~\eqref{eqn:subsolvernormofgrad}}, where $s_k = \sum_{i = 1}^k y^{(k)}_i q_i^{}$
		\EndWhile
		\Statex
		\State \textbf{Output:} $s_k \in \calX$
	\end{algorithmic}
	
\end{algorithm}


\section{Numerical experiments} \label{sec:XP}

We implement Algorithm~\ref{algo:ARC} within the Manopt framework~\citep{manopt} (our code is part of that toolbox) and compare the performance of our implementation against some existing solvers in that toolbox, namely, the Riemannian trust-region method (RTR)~\citep{genrtr} and the Riemannian conjugate gradients method with Hestenes--Stiefel update formula (CG-HS)~\citep[\S8.3]{AMS08}. All algorithms terminate when $\|\grad f(x_k)\| \leq 10^{-9}$. CG-HS also terminates if it is unable to produce a step of size more than $10^{-10}$. Code to reproduce the experiments is available at \url{https://github.com/NicolasBoumal/arc}. We report results with randomness fixed by \texttt{rng(2019)} within Matlab R2019b.

We consider a suite of six Riemannian optimization problems:
\begin{enumerate}
	\item Dominant invariant subspace: $\max_{X \in \Gr(n, p)} \frac{1}{2} \Trace(X\transpose A X)$, where $A \in \Rnn$ is symmetric (randomly generated from i.i.d.\ Gaussian entries) and $\Gr(n, p)$ is the Grassmann manifold of subspaces of dimension $p$ in $\Rn$, represented by orthonormal matrices in $\Rnp$. Optima correspond to dominant invariant subspaces of $A$~\citep{edelman1998geometry}.
	\item Truncated SVD: $\max_{U\in\St(m, p), V\in\St(n, p)} \Trace(U\transpose A V N)$, where $\St(n, p)$ is the set of matrices in $\Rnp$ with orthonormal columns, $A \in \Rmn$ has i.i.d.\ random Gaussian entries and $N = \diag(p, p-1, \ldots, 1)$. Global optima correspond to the $p$ dominant left and right singular vectors of $A$~\citep{sato2013riemannianSVD}. For this and the previous problem, the random matrices have small eigen or singular value gap, which makes them challenging.
	\item Low-rank matrix completion via optimization on one Grassmann manifold, as in~\citep{boumal2011rtrmc}. The target matrix $A \in \Rmn$ has rank $r$: it is fully specified by $r(m+n-r)$ parameters. We observe this many entries of $A$ picked uniformly at random, times an oversampling factor (osf). The task is to recover $A$ from those samples. $A$ is generated from two random Gaussian factors as $A = LR\transpose$ to have rank $r$ exactly. The variable is $U \in \Gr(m, r)$, and the cost function minimizes the sum of squared errors between $UW_U$ and the observed entries of $A$, where $W_U \in \reals^{r \times n}$ is the optimal matrix for that purpose (which has an explicit expression once $U$ is fixed, efficiently computable).
	\item Max-cut: given the adjacency matrix $A \in \Rnn$ of a graph with $n$ nodes, we solve the semidefinite relaxation of the Max-Cut graph partitioning problem via the Burer--Monteiro formulation~\citep{burer2005local} on the oblique manifold~\citep{journee2010low}: $\min_{X\in\mathrm{OB}(n, p)} \frac{1}{2}\Trace(X\transpose A X)$, where $\mathrm{OB}(n, p)$ is the set of matrices in $\Rnp$ with unit-norm rows. Here, we pick graph \#22 from a collection of graphs called Gset (see any of the references): it has $n = 2000$ nodes and $19990$ edges, and $p$ is set close to $\sqrt{2n}$ as justified in~\citep{boumal2018deterministicbm}.
	\item Synchronization of rotations: $m$ rotation matrices $Q_1, \ldots, Q_m$ in the special orthogonal group $\mathrm{SO}(d)$ are estimated from noisy relative measurements $H_{ij} \approx Q_i^{}Q_j\transpose$ for an Erd\H{o}s--R\'enyi random set of pairs $(i, j)$ following a maximum likelihood formulation, as in~\citep{boumal2013MLE}. The specific distribution of the measurements and the corresponding cost function are described in the reference. All algorithms are initialized with the technique proposed in the reference, to avoid convergence to a poor local optimum.
	\item ShapeFit: least-squares formulation of the problem of recovering a rigid structure of $n$ points $x_1, \ldots, x_n$ in $\Rd$ from noisy measurements of some of the pairwise directions $\frac{x_i - x_j}{\|x_i - x_j\|}$ picked uniformly at random, following~\citep{hand2018shapefit}. The set of points is centered and obeys one extra linear constraint to fix scaling ambiguity, so that the search space is effectively a linear subspace of $\reals^{n\times d}$: this is the manifold $\calM$. The cost function as spelled out in the reference makes this a structured linear least-squares problem.
\end{enumerate}

For each problem, we generate one instance and one random initial guess (except for problem 5 which is initialized deterministically). Then, we run each algorithm from that same initial guess on that same instance. Figure~\ref{fig:XPtimegradnorm} displays the progress of each algorithm on each problem as the gradient norm of iterates (on a log scale) as a function of elapsed computation time to reach each iterate (in seconds) on a laptop from 2016. For the same run, Figure~\ref{fig:XPhesscallsgradnorm} reports the number of gradient calls and Hessian-vector products (summed) issued by all algorithms along the way. Figure~\ref{fig:XPouteriterations} reports the number of outer iterations for ARC and RTR, that is, excluding work done by subsolvers.

For ARC, we report results with $\theta = 0.25$ and $\theta = 2$; this is used in the stopping criterion for subproblem solves following~\eqref{eq:firstorderprogress} (we do not check~\eqref{eq:secondorderprogress}). To initialize $\varsigma_0$, we use 100 divided by the initial trust-region radius of RTR ($\Delta_0$) chosen by Manopt. Other parameters of ARC are set as follows: $\varsigma_{\min} = 10^{-10}, \eta_1 = 0.1, \eta_2 = 0.9, \gamma_1 = 0.1, \gamma_2 = \gamma_3 = 2$, with update rule:
\begin{align}
\varsigma_{k+1} =
\begin{cases}
\begin{aligned}
& \max(\varsigma_{\min}, \gamma_1 \varsigma_k) & & \textrm{ if } \rho_k \geq \eta_2 & & \textrm{ (very successful),} \\
& \varsigma_k & & \textrm{ if } \rho_k \in [\eta_1, \eta_2) & & \textrm{ (successful),} \\
& \gamma_2 \varsigma_k & & \textrm{ if } \rho_k < \eta_1 & & \textrm{ (unsuccessful).}
\end{aligned}
\end{cases}
\end{align}
Standard safeguards to account for numerical round-off errors are included in the code (not described here). Parameters for the other methods have the default values given by Manopt.

We experiment with two subproblem solvers for ARC. The first one is the Lanczos-based method as described in Section~\ref{sec:subproblem} (ARC Lanczos). The second one is a (Euclidean) nonlinear, nonnegative-Polak--Ribi\`ere conjugate gradients method run on the model $m_k$~\eqref{eq:mk} in the tangent space $\T_{x_k}\calM$ (ARC NLCG), implemented by Bryan Zhu~\citep{zhu2019arcthesis}. This solver uses the initialization recommended by~\citet{carmon2016gradient} for gradient descent (and from which they proved convergence to a global optimizer, despite non-convexity of the model), and exact line-search.
We find that this subproblem solver performs well in practice. It is simpler to implement, and uses less memory than the Lanczos method.

\begin{figure}[p]
	\centering
	\includegraphics[width=1\linewidth]{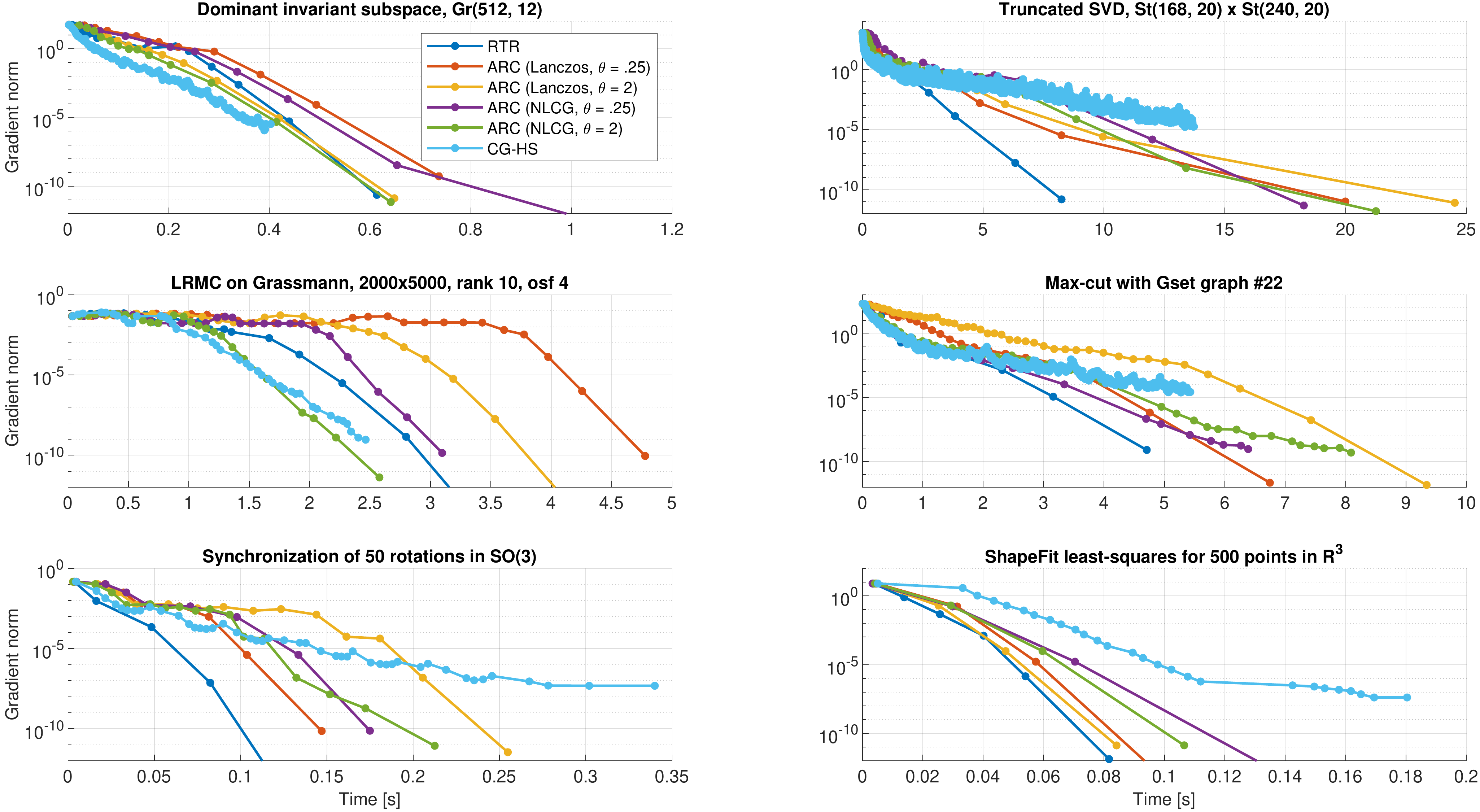}
	\caption{Gradient norm at each iterate for the three competing solvers on the six benchmark problems, as a function of computation time needed by those solvers to reach those iterates (in seconds). Our algorithm is ARC (tested with two different subproblem solvers, each with two parameter settings); RTR is Riemannian trust-regions; CG-HS is Riemannian conjugate gradients with Hestenes--Stiefel step selection.}
	\label{fig:XPtimegradnorm}
\end{figure}
\begin{figure}[p]
	\centering
	\includegraphics[width=1\linewidth]{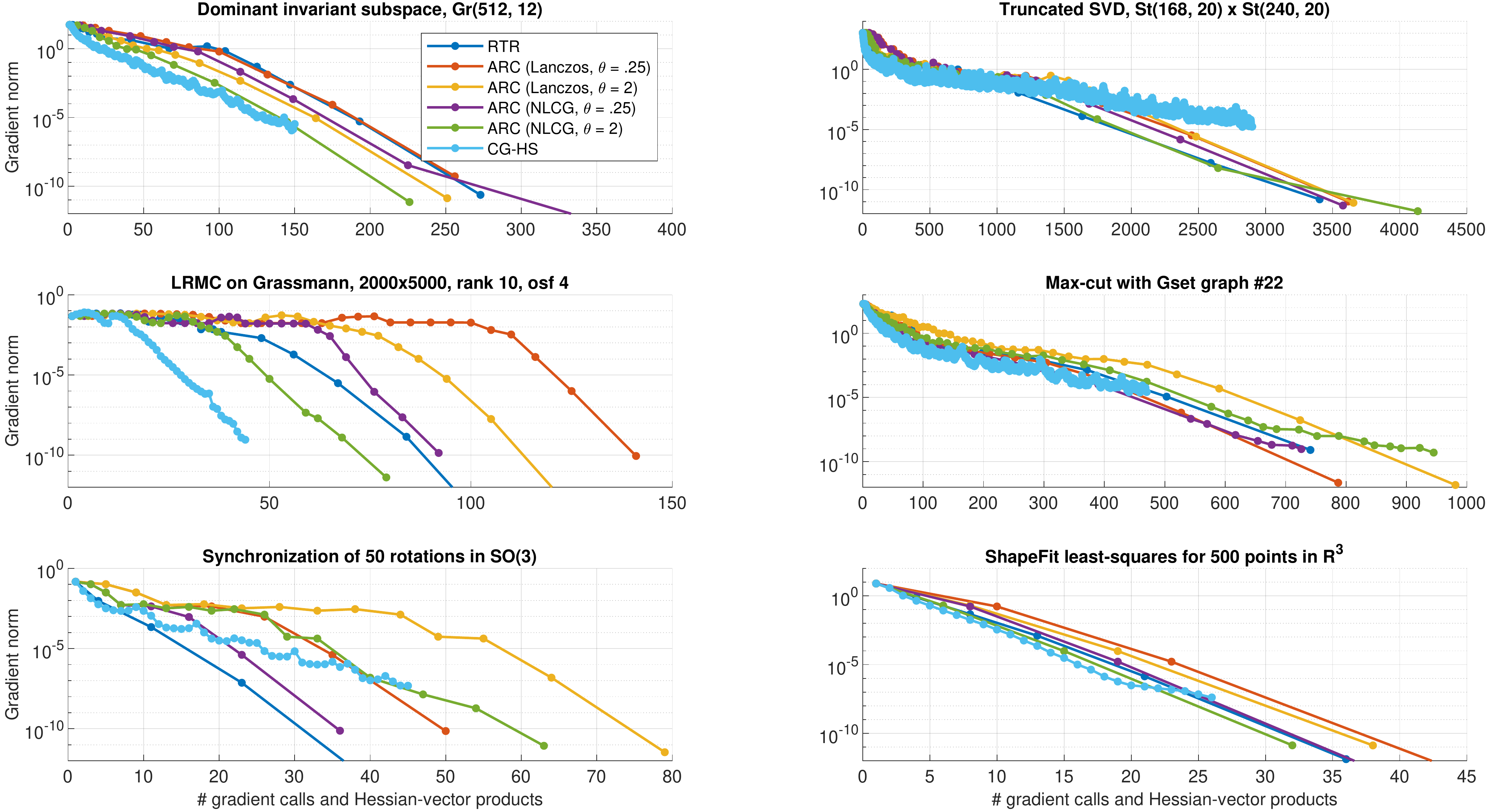}
	\caption{Gradient norm at each iterate, as a function of the number of gradient calls and Hessian-vector calls (the sum of both) issued by those solvers to reach those iterates.}
	\label{fig:XPhesscallsgradnorm}
\end{figure}
\begin{figure}[p]
	\centering
	\includegraphics[width=1\linewidth]{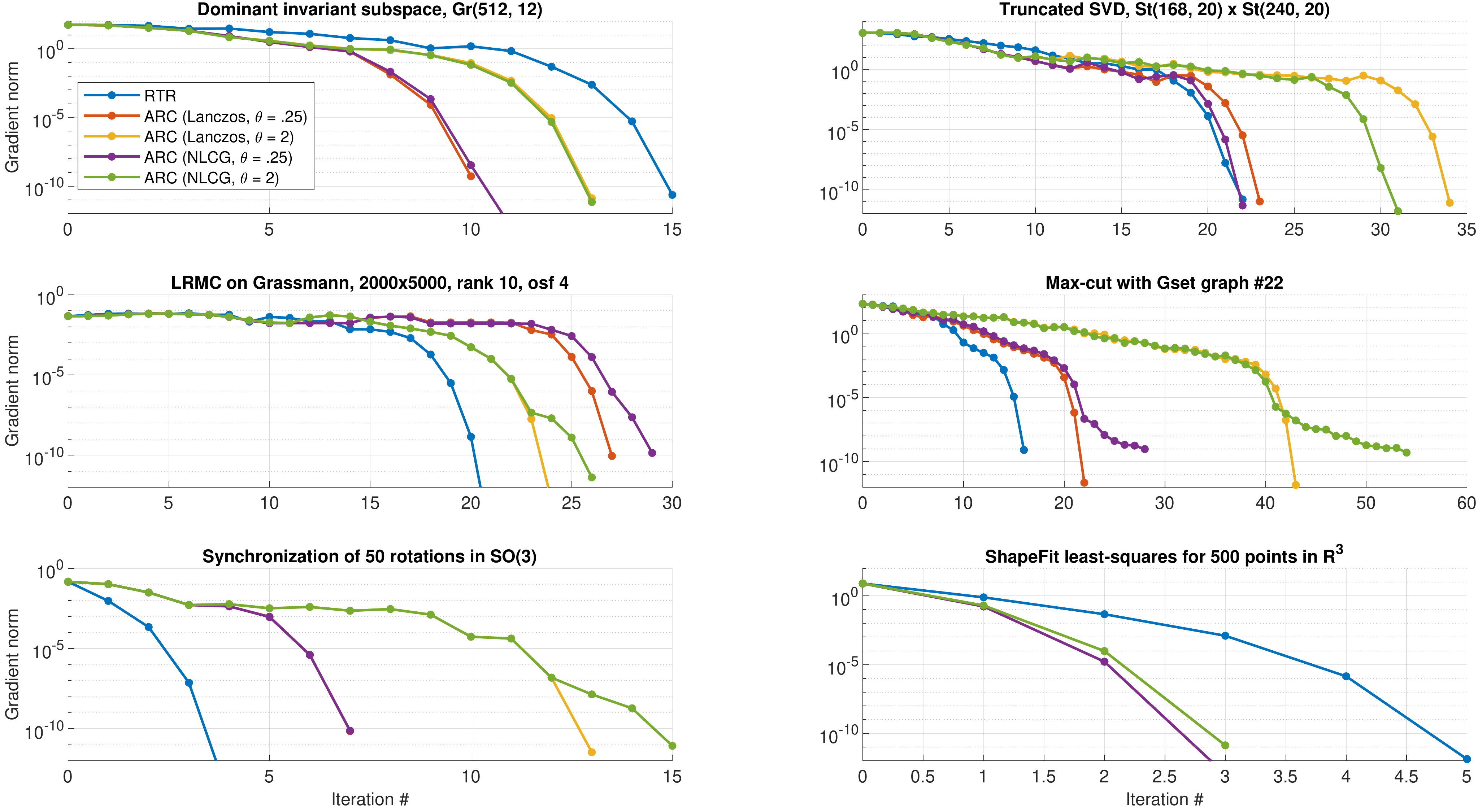}
	\caption{Gradient norm at each iterate, as a function of the number of outer iterations for ARC and RTR: both of these solvers rely on a subproblem solver. This plot compares the behavior of the algorithms separately from their subproblem solvers' work. As a result, this hides effects related to how stringent the stopping criterion of the subproblem solver is, hence of how costly the subproblem solves are. For example, one can (usually) reduce the number of outer iterations of ARC by reducing $\theta$. As the subproblems of RTR and ARC are similar, we expect that (in principle) it should be possible to solve them equally well in about the same time.}
	\label{fig:XPouteriterations}
\end{figure}

We find that ARC's performance is in the same ballpark as RTR's, with the caveat that ARC's best performance requires tuning (choosing the right subproblem solver and $\theta$ for the problem class), whereas RTR is more robust. Since RTR's code has been refined over many years, we expect that further work can help reduce the gap. For example, we expect that the performance of ARC could be improved with further tuning of the regularization parameter update rule. In particular, we find that it is important to reduce regularization fast when close to convergence (but not earlier), to allow ARC to make steps similar to Newton's method. Work by~\citet{gould2012updatingregularization} could be a good starting point for such exploration.





\paragraph{Acknowledgments}
We thank Pierre-Antoine Absil for numerous insightful and technical discussions, Stephen McKeown for directing us to, and guiding us through the relevance of Jacobi fields for our study of~\aref{assu:DRetr}, Chris Criscitiello and Eitan Levin for many discussions regarding regularity assumptions on manifolds, and Bryan Zhu for contributing his nonlinear CG subproblem solver to Manopt, and related discussions.


\bibliographystyle{abbrvnat}
\bibliography{refs}

\clearpage

\appendix 

\section{Proofs from Section~\ref{sec:arc}: mechanical lemmas} \label{app:arc}

Lemma~\ref{lem:termination} characterizes the conditions under which the subproblem solver is allowed to return $s_k = 0$ at iteration $k$.
\begin{proof}[Proof of Lemma~\ref{lem:termination}]
	By definition of the model $m_k$~\eqref{eq:mk} and by properties of retractions~\eqref{eq:nablahatfTgradf},
	\begin{align*}
	\nabla m_k(0) & = \nabla \hat f_k(0) = \grad f(x_k),
	\end{align*}
	where $\hat f_k = f \circ \Retr_{x_k}$.
	Thus, if $\grad f(x_k) = 0$, the first-order condition~\eqref{eq:firstorderprogress} allows $s_k = 0$. The other way around, if $s_k = 0$ is allowed, then $\|\nabla m_k(0)\| = 0$, so that $\grad f(x_k) = 0$.
	
	Now assume the second-order condition~\eqref{eq:secondorderprogress} is enforced. If $s_k = 0$ is allowed, then we already know that $\grad f(x_k) = 0$. Combined with~\eqref{eq:nablatwohatfHessfWzero}, we deduce that
	\begin{align*}
	\nabla^2 m_k(0) = \nabla^2 \hat f_k(0) = \Hess f(x_k),
	\end{align*}
	for any retraction. Then, condition~\eqref{eq:secondorderprogress} at $s_k = 0$ indicates $\nabla^2 m_k(0)$ is positive semidefinite, hence $\Hess f(x_k)$ is positive semidefinite. The other way around, if $\grad f(x_k) = 0$ and $\Hess f(x_k)$ is positive semidefinite, then $\nabla m_k(0) = \grad f(x_k)$ and $\nabla^2 m_k(0) = \Hess f(x_k)$, so that indeed $s_k = 0$ is allowed.
\end{proof}

The two supporting lemmas presented in Section~\ref{sec:arc} follow from the regularization parameter update mechanism of Algorithm~\ref{algo:ARC}. The standard proofs are not affected by the fact we here work on a manifold. We provide them for the sake of completeness.
\begin{proof}[Proof of Lemma~\ref{lem:varsigmamax}]
	 Using the definition of $\rho_k$~\eqref{eq:rhok}, $m_k(0) = f(x_k)$~\eqref{eq:mk} and $m_k(0) - m_k(s_k) \geq 0$ by condition~\eqref{eq:firstorderprogress}: 
	\begin{align*}
		1 - \rho_k & = 1 - \frac{f(x_k) - f(\Retr_{x_k}(s_k))}{m_k(0) - m_k(s_k) + \frac{\varsigma_k}{3} \|s_k\|^3} \leq  \frac{f(\Retr_{x_k}(s_k)) - m_k(s_k) + \frac{\varsigma_k}{3} \|s_k\|^3}{\frac{\varsigma_k}{3} \|s_k\|^3}.
	\end{align*}
	Owing to~\aref{assu:firstorderregularityscalar}, the numerator is upper bounded by $(L/6)\|s_k\|^3$. Hence, $1 - \rho_k \leq \frac{L}{2\varsigma_k}$. If $\varsigma_k \geq \frac{L}{2(1-\eta_2)}$, then $1-\rho_k \leq 1-\eta_2$ so that $\rho_k \geq \eta_2$, meaning step $k$ is very successful. The regularization mechanism~\eqref{eq:varsigmaupdate} then ensures $\varsigma_{k+1} \leq \varsigma_k$.
	Thus, $\varsigma_{k+1}$ may exceed $\varsigma_k$ only if $\varsigma_k < \frac{L}{2(1-\eta_2)}$, in which case it can grow at most to $\frac{L\gamma_3}{2(1-\eta_2)}$, but cannot grow beyond that level in later iterations.
%
\end{proof}

\begin{proof}[Proof of Lemma~\ref{lem:boundKsuccessfulsteps}]
	Partition iterations $0, \ldots, \bar k - 1$ into successful or very successful ($\calS_{\bar k}$) and unsuccessful ($\calU_{\bar k}$) ones.
	Following the update mechanism~\eqref{eq:varsigmaupdate}, for $k \in \calS_{\bar k}$, $\varsigma_{k+1} \geq \gamma_1 \varsigma_k$, while for $k \in \calU_{\bar k}$, $\varsigma_{k+1} \geq \gamma_2 \varsigma_k$. Thus, by induction, $\varsigma_{\bar k} \geq \varsigma_0 \gamma_1^{|\calS_{\bar k}|} \gamma_2^{|\calU_{\bar k}|}$. By assumption, $\varsigma_{\bar k} \leq \varsigma_{\max}$ so that
	\begin{align*}
		\log\left(\frac{\varsigma_{\max}}{\varsigma_0}\right) \geq |\calS_{\bar k}| \log(\gamma_1) + |\calU_{\bar k}| \log(\gamma_2) = |\calS_{\bar k}|\left[\log(\gamma_1) - \log(\gamma_2)\right] + \bar k \log(\gamma_2),
	\end{align*}
	where we also used $|\calS_{\bar k}| + |\calU_{\bar k}| = \bar k$. Isolating $\bar k$ using $\gamma_2 > 1 > \gamma_1$ 
	allows to conclude.
\end{proof}

\section{Proofs from Section~\ref{sec:firstorderanalysisexp}: first-order analysis, exponentials} \label{app:firstorderexp}

Certain tools from Riemannian geometry are useful throughout the appendices---see for example \cite[pp59--67]{oneill}. To fix notation, let $\nabla$ denote the Riemannian connection on $\calM$ (not to be confused with $\nabla$ and $\nabla^2$ which denote gradient and Hessian of functions on linear spaces, such as pullbacks). With this notation, the Riemannian Hessian~\cite[Def.~5.5.1]{AMS08} is defined by $\Hess f = \nabla \grad f$. Furthermore, $\frac{\D}{\dt}$ denotes the covariant derivative of vector fields along curves on $\calM$, induced by $\nabla$. With this notation, given a smooth curve $c \colon \reals \to \calM$, the intrinsic acceleration is defined as $c''(t) = \frac{\D^2}{\dt^2} c(t)$. For example, for a Riemannian submanifold of a Euclidean space, $c''(t)$ is obtained by orthogonal projection of the classical acceleration of $c$ in the embedding space to the tangent space at $c(t)$. Geodesics are those curves which have zero intrinsic acceleration.


We first state and prove a partial version of Proposition~\ref{prop:LipschitzHessianBounds} which applies for general retractions. Right after this, we prove Proposition~\ref{prop:LipschitzHessianBounds}. The purpose of this detour is to highlight how crucial properties of geodesics and of their interaction with parallel transports allow for the more direct guarantees of Section~\ref{sec:firstorderanalysisexp}. In turn, this serves as motivation for the developments in Section~\ref{sec:firstorderanalysisgeneral}.
\begin{proposition} \label{prop:LipschitzHessianBoundsretraction}
	Let $f \colon \calM \to \reals$ be twice differentiable on a Riemannian manifold $\calM$ equipped with a retraction $\Retr$. Given $(x, s) \in \T\calM$, assume there exists $L \geq 0$ such that, for all $t \in [0, 1]$,
	\begin{align*}
		\left\| P_{t s}^{-1}\!\left(\Hess f(c(t))[c'(t)]\right) - \Hess f(x)[s] \right\| & \leq L \|s\| \cdot \ell(c|_{[0, t]}),
	\end{align*}
	where $P_{ts}$ is parallel transport along $c(t) = \Retr_x(ts)$ from $c(0)$ to $c(t)$ (note the retraction instead of the exponential) 
	and $\ell(c|_{[0, t]}) = \int_{0}^{t} \|c'(\tau)\| \dtau$ is the length of $c$ restricted to the interval $[0, t]$. Then,
	\begin{align*}
		\left\| P_{s}^{-1}\grad f(\Retr_x(s)) - \grad f(x) - \Hess f(x)[s] \right\| & \leq L \|s\| \int_{0}^{1} \ell(c|_{[0, t]}) \,\dt.
	\end{align*}	
\end{proposition}
\begin{proof}
	Pick a basis $v_1, \ldots, v_d$ for $\T_x\calM$, and define the parallel vector fields $V_i(t) = P_{ts}(v_i)$ along $c(t)$. Since parallel transport is an isometry, $V_1(t), \ldots, V_d(t)$ form a basis for $\T_{c(t)}\calM$ for each $t \in [0, 1]$. As a result, we can express the gradient of $f$ along $c(t)$ in these bases,
	\begin{align}
		\grad f(c(t)) & = \sum_{i = 1}^{d} \alpha_i(t) V_i(t),
		\label{eq:gradfctexpansion}
	\end{align}
	with $\alpha_1(t), \ldots, \alpha_d(t)$ differentiable. Using properties of the Riemannian connection $\nabla$ and its associated covariant derivative $\Ddt$~\cite[pp59--67]{oneill}, we find on one hand that
	\begin{align*}
		\Ddt \grad f(c(t)) & = \nabla_{c'(t)} \grad f = \Hess f(c(t))[c'(t)],
	\end{align*}
	and on the other hand that
	\begin{align*}
		\Ddt \sum_{i = 1}^{d} \alpha_i(t) V_i(t) & = \sum_{i = 1}^{d} \alpha_i'(t) V_i(t) = P_{ts} \sum_{i = 1}^{d} \alpha_i'(t) v_i,
	\end{align*}
	where we used that $\Ddt V_i(t) = 0$, by definition of parallel transport.
	Furthermore,
	\begin{align*}
		c'(t) = \D\Retr_x(ts)[s] = T_{ts}(s),
	\end{align*}
	where $T_{ts} = \D\Retr_x(ts)$ is a linear operator from the tangent space at $x$ to the tangent space at $c(t)$---just like $P_{ts}$.
	Combining, we deduce that
	\begin{align*}
		\sum_{i = 1}^d \alpha_i'(t) v_i & = \left(P_{ts}^{-1} \circ \Hess f(c(t)) \circ T_{ts}\right)\![s].
	\end{align*}
	Going back to~\eqref{eq:gradfctexpansion}, we also see that
	\begin{align*}
		G(t) & \triangleq P_{ts}^{-1} \grad f(c(t)) = \sum_{i = 1}^{d} \alpha_i(t) v_i
	\end{align*}
	is a map from (a subset of) $\reals$ to $\T_x\calM$---two linear spaces---so that we can differentiate it in the usual way:
	\begin{align*}
		G'(t) = \sum_{i = 1}^{d} \alpha_i'(t) v_i.
	\end{align*}
	We conclude that
	\begin{align}
		G'(t) = \ddt\!\left[ P_{ts}^{-1} \grad f(c(t)) \right] & = \left(P_{ts}^{-1} \circ \Hess f(c(t)) \circ T_{ts}\right)\![s].
		\label{eq:ddtptsinvgradfct}
	\end{align}
	Since $G'$ is continuous,
	\begin{align*}
		P_{ts}^{-1}\grad f(c(t)) = G(t) & = G(0) + \int_{0}^{t} G'(\tau) \dtau \\
										& = \grad f(x) + \int_{0}^{t} \left(P_{\tau s}^{-1} \circ \Hess f(c(\tau)) \circ T_{\tau s}\right)\![s]  \dtau.
	\end{align*}
	Moving $\grad f(x)$ to the left-hand side and subtracting $\Hess f(x)[ts]$ on both sides, we find
	\begin{align*}
		P_{ts}^{-1}\grad f(c(t)) - \grad f(x) - \Hess f(x)[ts] & = \int_{0}^{t} \left(P_{\tau s}^{-1} \circ \Hess f(c(\tau)) \circ T_{\tau s} - \Hess f(x)\right)\![s]  \dtau.
	\end{align*}
	Using the main assumption on $\Hess f$ along $c$, it easily follows that
	\begin{align}
		\left\| P_{ts}^{-1}\grad f(c(t)) - \grad f(x) - \Hess f(x)[ts] \right\| & \leq \|s\| L \int_{0}^{t} \ell(c|_{[0, \tau]}) \dtau.
		\label{eq:step2normboundLipschitzHessiangeneral}
	\end{align}
	For $t = 1$, this is the announced inequality.
\end{proof}

\begin{proof}[Proof of Proposition~\ref{prop:LipschitzHessianBounds}]
	In this proposition we work with the exponential retraction, so that instead of a general retraction curve $c(t)$ we work along a geodesic $\gamma(t) = \Exp_x(ts)$. By definition, the velocity vector field $\gamma'(t)$ of a geodesic $\gamma(t)$ is parallel, meaning
	\begin{align}
		\gamma'(t) = P_{ts}(\gamma'(0)) = P_{ts}(s).
		\label{eq:gammaprimePt}
	\end{align}
	This elegant interplay of geodesics and parallel transport is crucial.
	In particular, 
	\begin{align*}
		 \ell(\gamma|_{[0, t]}) = \int_{0}^{t} \|\gamma'(\tau)\| \dtau = t \|s\|,
	\end{align*}
	and the condition in Proposition~\ref{prop:LipschitzHessianBoundsretraction} becomes
	\begin{align*}
		\left\| P_{t s}^{-1}\!\left(\Hess f(\gamma(t))[P_{ts}(s)]\right) - \Hess f(x)[s] \right\| & \leq t L \|s\|^2,
	\end{align*}
	which is indeed guaranteed by our own assumptions. We deduce that~\eqref{eq:step2normboundLipschitzHessiangeneral} holds:
	\begin{align}
		\left\| P_{ts}^{-1}\grad f(\gamma(t)) - \grad f(x) - \Hess f(x)[ts] \right\| & \leq \|s\| L \int_{0}^{t} \ell(\gamma|_{[0, \tau]}) \dtau = \frac{L}{2} \|s\|^2 t^2.
		\label{eq:step2normboundLipschitzHessian}
	\end{align}
	The relation~\eqref{eq:gammaprimePt} also yields the scalar inequality. Indeed, since $f \circ \gamma \colon [0, 1] \to \reals$ is continuously differentiable,
	\begin{align*}
	f(\Exp_x(s)) = f(\gamma(1)) & = f(\gamma(0)) + \int_{0}^{1} (f \circ \gamma)'(t) \dt \\
							& = f(x) + \int_{0}^{1} \inner{\grad f(\gamma(t))}{\gamma'(t)} \dt \\
							& = f(x) + \int_{0}^{1} \inner{P_{ts}^{-1}\grad f(\gamma(t))}{s} \dt,
	\end{align*}
	where on the last line we used~\eqref{eq:gammaprimePt} and the fact that $P_{ts}$ is an isometry.
	For a general retraction curve $c(t)$, instead of $s$ as the right-most term we would find $P_{ts}^{-1}(c'(t))$ which may vary with $t$: this would make the next step significantly more difficult.
	Move $f(x)$ to the left-hand side and subtract terms on both sides to get
	\begin{multline*}
		f(\Exp_x(s)) - f(x) - \inner{\grad f(x)}{s} - \frac{1}{2} \inner{s}{\Hess f(x)[s]} \\ = \int_{0}^{1} \inner{P_{ts}^{-1}\grad f(\gamma(t)) - \grad f(x) - \Hess f(x)[ts]}{s} \dt.
	\end{multline*}
	Using~\eqref{eq:step2normboundLipschitzHessian} and Cauchy--Schwarz, it follows immediately that
	\begin{align*}
	\left| f(\Exp_x(s)) - f(x) - \inner{s}{\grad f(x)} - \frac{1}{2} \inner{s}{\Hess f(x)[s]} \right| \leq \int_{0}^{1} \frac{L}{2} \|s\|^3 t^2 \dt = \frac{L}{6} \|s\|^3,
	\end{align*}
	as announced.
\end{proof}

Next, we provide an argument for the last claim in Theorem~\ref{thm:masterboundSLipschitzHessian}.
\begin{proof}[Proof of Theorem~\ref{thm:masterboundSLipschitzHessian}]
	We argue that $\lim_{k \to \infty} \|\grad f(x_k)\| = 0$. The first claim of the theorem states that, for every $\varepsilon > 0$, there is a finite number of successful steps $k$ such that $x_{k+1}$ has gradient larger than $\varepsilon$. Thus, for any $\varepsilon > 0$, there exists $K$: the last successful step such that $x_{K+1}$ has gradient larger than $\varepsilon$. Furthermore, there is a finite number of unsuccessful steps directly after $K+1$. Indeed, $\varsigma_{K+1} \geq \varsigma_{\min}$, and failures increase $\varsigma$ exponentially; additionally, $\varsigma$ cannot outgrow $\varsigma_{\max}$ by Lemma~\ref{lem:varsigmamax}. Thus, after a finite number of failures, a new success arises, necessarily producing an iterate with gradient norm at most $\varepsilon$ since $K$ was the last successful step to produce a larger gradient. By the same argument, all subsequent iterates have gradient norm at most $\varepsilon$. In other words: for any $\varepsilon > 0$, there exists $K'$ finite such that for all $k \geq K'$, $\|\grad f(x_k)\| \leq \varepsilon$, that is: $\lim_{k\to\infty} \|\grad f(x_k)\| = 0$.
\end{proof}



\section{Proofs from Section~\ref{sec:secondorder}: second-order analysis} \label{app:secondorder}

\begin{proof}[Proof of Corollary~\ref{cor:mastercorollarysecond}]
	Consider these subsets of the set of successful iterations $\calS$: 
	\begin{align*}
		\calS^1 & \defeq \{k \in \calS : \|\grad f(x_{k + 1})\| > \varepsilon_g\}, & & \text{ and } & \calS^2 & \defeq \{k \in \calS : \lambda_{\min}(\Hess f(x_{k})) < - \varepsilon_{{H}}\}.
	\end{align*}
	These sets are finite: for $K_1 = K_1(\varepsilon_g)$ as provided by either Theorem~\ref{thm:masterboundSLipschitzHessian} or Theorem~\ref{thm:masterboundSLipschitzHessianretraction}, and for $K_2 = K_2(\varepsilon_{{H}})$ as provided by Theorem~\ref{thm:masterboundSsecond}, we know that
	\begin{align*}
	|\calS^1| & \leq K_1, & & \textrm{ and } & |\calS^2| & \leq K_2.
	\end{align*}
	Note that successful steps are in one-to-one correspondence with the distinct points in the sequence of iterates $x_1, x_2, x_3, \ldots$\footnote{This is true because the cost function is strictly decreasing when successful, so that any $x_k$ can only be repeated in one contiguous subset of iterates. Hence, if $k$ is a successful iteration, match it to $x_{k+1}$ (this is why we omitted $x_0$ from the list.)} The first inequality states at most $K_1$ of the distinct points in that list have large gradient. The second inequality states at most $K_2$ of the distinct points in that same list have significantly negative Hessian eigenvalues. Thus, if more than $K_1+K_2+1$ distinct points appear among $x_0, x_1, \ldots, x_{\bar k}$ (note the $+1$ as we added $x_0$ to the list),
	then at least one of these points has both a small gradient and an almost positive semidefinite Hessian.
	%
	In particular, as long as the number of successful iterations among $0, \ldots, \bar k-1$ exceeds $K_1 + K_2 + 1$ (strictly),
	there must exist $k \in \{0, \ldots, \bar k\}$ such that
	\begin{align*}
	\|\grad f(x_{k})\| & \leq \varepsilon_g & & \text{ and } & \lambda_{\min}(\Hess f(x_{k})) & \geq - \varepsilon_{{H}}.
	\end{align*} 
	Lemma~\ref{lem:boundKsuccessfulsteps} allows to conclude. 
	%
	%
\end{proof}

\section{Proofs from Section~\ref{sec:regularity}: regularity assumptions} \label{app:regularity}
\begin{proof}[Proof of Lemma~\ref{lem:lipschitzhessianbasic}]
	Since $\hat f$ is a real function on a linear space, standard calculus applies:
	\begin{align*}
		\hat f(s) - \left[ \hat f(0) + \innersmall{s}{\nabla \hat f(0)} + \frac{1}{2} \innersmall{s}{\nabla^2 \hat f(0)[s]} \right] & = \int_{0}^{1}\!\int_{0}^{1} \tauort_1 \inner{\left[ \nabla^2 \hat f(\tauort_1\tauort_2 s) - \nabla^2 \hat f(0) \right]\![s]}{s}  \mathrm{d}\tauort_1 \mathrm{d}\tauort_2, \\
		\nabla \hat f(s) - \left[\nabla \hat f(0) + \nabla^2 \hat f(0)[s]\right] & = \int_{0}^{1} \left[ \nabla^2 \hat f(\tauort s) - \nabla^2 \hat f(0) \right]\![s] \, \mathrm{d}\tauort.
	\end{align*}
	Taking norms on both sides, by a triangular inequality to pass the norm through the integral and integrating respectively $\tauort_1^2 \tauort_2^{}$ and $\tauort$, we find using our main assumption~\eqref{eq:lipschitzhessianbasic} that
	\begin{align*}
		\left| \hat f(s) - \left[ \hat f(0) + \innersmall{s}{\nabla \hat f(0)} + \frac{1}{2} \innersmall{s}{\nabla^2 \hat f(0)[s]} \right] \right| & \leq \frac{1}{6} L \|s\|^3, \textrm{ and } \\
		\left\| \nabla \hat f(s) - \left[\nabla \hat f(0) + \nabla^2 \hat f(0)[s]\right] \right\| & \leq \frac{1}{2} L \|s\|^2. \qedhere
	\end{align*}
\end{proof}

\begin{proof}[Proof of Lemma~\ref{lem:derivativespullback}]
For an arbitrary $\dot s \in \T_x\calM$, consider the curve $c(t) = \Retr_x(s + t\dot s)$, and let $g = f \circ c \colon \reals \to \reals$. We compute the derivatives of $g$ in two different ways. On the one hand, $g(t) = \hat f(s + t\dot s)$ so that
	\begin{align*}
		g'(t) & = \D\hat f(s+t\dot s)[\dot s] = \innersmall{\nabla \hat f(s+t\dot s)}{\dot s}, \\
		g''(t) & = \inner{ \ddt \nabla \hat f(s+t\dot s)}{\dot s} = \innersmall{\nabla^2 \hat f(s+t\dot s)[\dot s]}{\dot s}.
	\end{align*}
	On the other hand, $g(t) = f(c(t))$ so that, using properties of $\frac{\D}{\mathrm{d}t}$~\cite[pp59--67]{oneill}:
	\begin{align*}
		g'(t)  & = \D f(c(t))[c'(t)] = \innersmall{\grad f(c(t))}{c'(t)}, \\
		g''(t) & = \ddt \inner{(\grad f \circ c)(t))}{c'(t)} \\
		       & = \inner{\nabla_{c'(t)} \grad f}{c'(t)} + \inner{(\grad f \circ c)(t))}{\frac{\D}{\dt} c'(t)} \\
		 	   & = \inner{\Hess f(c(t))[c'(t)]}{c'(t)} + \inner{\grad f(c(t))}{c''(t)}.
	\end{align*}
%
	Equating the different identities for $g'(t)$ and $g''(t)$ at $t = 0$ while using $c'(0) = T_s \dot s$, we find for all $\dot s \in \T_x\calM$:
	\begin{align*}
		\innersmall{\nabla \hat f(s)}{\dot s} & = \inner{\grad f(\Retr_x(s))}{T_s \dot s}, \\
		\innersmall{\nabla^2 \hat f(s)[\dot s]}{\dot s} & = \inner{\Hess f(\Retr_x(s))[T_s \dot s]}{T_s \dot s} + \inner{\grad f(\Retr_x(s))}{c''(0)}.
	\end{align*}
	The last term, $\inner{\grad f(\Retr_x(s))}{c''(0)}$, is seen to be the difference of two quadratic forms in $\dot s$, so that it is itself a quadratic form in $\dot s$. This justifies the definition of $W_s$ through polarization.
	The announced identities follow by identification.
\end{proof}

\begin{proof}[Proof of Proposition~\ref{prop:sphereretrnice}]
	With $\Retr_x(s) = \frac{x+s}{\sqrt{1+\|s\|^2}}$, it is easy to derive
	\begin{align}
		T_s \dot s \triangleq \D\Retr_x(s)[\dot s] & = \left[ \frac{1}{\sqrt{1+\|s\|^2}} I_n - \frac{1}{\sqrt{1+\|s\|^2}^3} (x+s)s\transpose \right] \dot s \nonumber \\
		& = \frac{1}{\sqrt{1+\|s\|^2}} \left[ I_n - \Retr_x(s) \Retr_x(s)\transpose \right] \dot s,
		\label{eq:DretrSphere}
	\end{align}
	where we used $x\transpose \dot s = 0$ in between the two steps to replace $s\transpose$ with $(x+s)\transpose$.
	The matrix between brackets is the orthogonal projector from $\Rn$ to $\T_{\Retr_x(s)}\calM$. Thus, its singular values are upper bounded by 1. Since $T_s$ is an operator on $\T_x\calM \subset \Rn$, 
	\begin{align*}
		\opnorm{T_s} & \leq \frac{1}{\sqrt{1+\|s\|^2}} \leq 1.
	\end{align*}
	This secures the first property with $c_1 = 1$.
	
	For the second property, consider $U(t) = T_{ts} \dot s$ and
	\begin{align*}
		 U'(t) \triangleq \frac{\D}{\mathrm{d}t} U(t) = \Proj_{c(t)} \ddt U(t),
	\end{align*}
	where $\Proj_y(v) = v - y (y\transpose v)$ is the orthogonal projector to $\T_y\calM$ and $c(t) = \Retr_x(ts)$. Define $g(t) = \frac{1}{\sqrt{1+t^2\|s\|^2}}$. Then, from~\eqref{eq:DretrSphere}, we have
	\begin{align}
		U(t) & = \left[ g(t) I_n - t g(t)^3 (x+ts)s\transpose \right] \dot s.
	\end{align}
	This is easily differentiated in the embedding space $\Rn$:
	\begin{align*}
		\ddt U(t) & = \left[ g'(t) I_n - (t g(t)^3)' (x+ts)s\transpose - t g(t)^3 ss\transpose \right] \dot s.
	\end{align*}
	The projection at $c(t)$ zeros out the middle term, as it is parallel to $x+ts$. This offers a simple expression for $U'(t)$, where in the last equality we use $g'(t) = -t g(t)^3 \|s\|^2$:
	\begin{align*}
		U'(t) & = \Proj_{c(t)}\left( \left[ g'(t) I_n - t g(t)^3 ss\transpose \right] \dot s \right) = -t g(t)^3 \cdot  \Proj_{c(t)}\left( \left[ \|s\|^2 I_n + ss\transpose \right] \dot s \right).
	\end{align*}
	The norm can only decrease after projection, so that, for $t \in [0, 1]$, 
	\begin{align*}
		\|U'(t)\| & \leq 2 t g(t)^3 \|s\|^2 \|\dot s\|.
	\end{align*}
	Let $h(t) = 2 t g(t)^3 \|s\|^2 = \frac{2t\|s\|^2}{(1+t^2 \|s\|^2)^{1.5}}$. For $s = 0$, $h$ is identically zero. Otherwise, $h$ attains its maximum $h\left(t = \frac{1}{\sqrt{2} \|s\|}\right) = \frac{4 \sqrt{3}}{9} \|s\|$. It follows that $\|U'(t)\| \leq c_2 \|s\| \|\dot s\|$ for all $t \in [0, 1]$ with $c_2 = \frac{4\sqrt{3}}{9}$.
	
	Finally, we establish the last property. Given $s, \dot s \in \T_x\calM$, consider $c(t) = \Retr_x(s + t\dot s)$. Simple calculations yield:
	\begin{align}
		c'(t) & = \ddt c(t) = \frac{1}{\sqrt{1 + \|s + t\dot s\|^2}} \left[ \dot s - \inner{\dot s}{c(t)} c(t) \right] = \frac{1}{\sqrt{1 + \|s + t\dot s\|^2}} \Proj_{c(t)} \dot s.
		\label{eq:ddtRsphere}
	\end{align}
	This is indeed in the tangent space at $c(t)$. The classical derivative of $c'(t)$ is given by
	\begin{align*}
		\ddt c'(t) & = - \frac{1}{\sqrt{1 + \|s + t\dot s\|^2}} \left[ \inner{\dot s}{c'(t)} c(t) + \inner{\dot s}{c(t)} c'(t) + \frac{\inner{s + t\dot s}{\dot s}}{1 + \|s + t \dot s\|^2} \Proj_{c(t)} \dot s \right] \\
			& = - \frac{1}{\sqrt{1 + \|s + t\dot s\|^2}} \left[ \inner{\dot s}{c'(t)} c(t) + 2\frac{\inner{s + t\dot s}{\dot s}}{1 + \|s + t \dot s\|^2} \Proj_{c(t)} \dot s \right],
	\end{align*}
	where we used~\eqref{eq:ddtRsphere} and orthogonality of $x$ and $\dot s$ in $\inner{c(t)}{\dot s} = \frac{1}{\sqrt{1+\|s+t\dot s\|^2}} \inner{x + s + t \dot s}{\dot s}$.
	The acceleration of $c$ is $c''(t) = \frac{\D}{\mathrm{d}t} c'(t) = \Proj_{c(t)}\left( \ddt c'(t) \right)$. The first term vanishes after projection, while the second term is unchanged. Overall,
	\begin{align}
		c''(t) & = -\frac{2\inner{s + t\dot s}{\dot s}}{\sqrt{1 + \|s + t \dot s\|^2}^3} \Proj_{c(t)} \dot s =  -\frac{2\inner{c(t)}{\dot s}}{1 + \|s + t \dot s\|^2} \Proj_{c(t)} \dot s.
	\end{align}
	In particular, $c''(0) = -2 \frac{\inner{s}{\dot s}}{\sqrt{1+\|s\|^2}^3} \Proj_{c(0)} \dot s$, so that $\|c''(0)\| \leq 2 \min(\|s\|, 0.4) \|\dot s\|^2$ and the property holds with $c_3 = 2$.
	(Peculiarly, if $s$ and $\dot s$ are orthogonal, $c''(0) = 0$.)
\end{proof}

In order to prove Theorem~\ref{thm:regularitycompact}, we introduce two supporting lemmas (needed only for the case where $\calM$ is not compact) and one key lemma. The first lemma below is similar in spirit to~\citep[Lem.~2.2]{cartis2011adaptivecubicPartI}.
\begin{lemma}\label{lem:boundessteps}
	Let $f \colon \calM \to \reals$ be twice continuously differentiable. 
	Let $\{ (x_0, s_0), (x_1, s_1), \ldots \}$ be the points and steps generated by Algorithm~\ref{algo:ARC}.
	Each step has norm bounded as:
	\begin{align}
		\|s_k\| & \leq \sqrt{\frac{3 \|\nabla \hat f_k(0)\|}{\varsigma_{\min}}} + \frac{3}{2\varsigma_{\min}} \max\left(0, -\lambdamin(\nabla^2 \hat f_k(0))\right),
	\end{align}
	where $\hat f_k = f \circ \Retr_{x_k}$ is the pullback, as in~\eqref{eq:fhatk}.
\end{lemma}
\begin{proof}
	Owing to the first-order progress condition~\eqref{eq:firstorderprogress}, using Cauchy--Schwarz and the fact that $\varsigma_k \geq \varsigma_{\min}$ for all $k$ by design of the algorithm, we find
	\begin{align*}
		\varsigma_{\min} \|s_k\|^3 \leq \varsigma_{k} \|s_k\|^3 & \leq -3\inner{s_k}{\nabla \hat f_k(0) + \frac{1}{2}\nabla^2 \hat f_k(0)[s_k]} \\
			& \leq 3 \|s_k\| \left( \|\nabla \hat f_k(0)\| + \frac{1}{2} \max\left(0, -\lambdamin(\nabla^2 \hat f_k(0))\right) \|s_k\| \right).
	\end{align*}
	This defines a quadratic inequality in $\|s_k\|$:
	\begin{align*}
		\varsigma_{\min} \|s_k\|^2 - h_k \|s_k\| - g_k \leq 0,
	\end{align*}
	where to simplify notation we let $h_k = \frac{3}{2} \max(0, -\lambdamin(\nabla^2 \hat f_k(0)))$ and $g_k = 3 \|\nabla \hat f_k(0)\|$. Since $\|s_k\|$ must lie between the two roots of this quadratic, we know in particular that
	\begin{align*}
		\|s_k\| & \leq \frac{h_k + \sqrt{h_k^2 + 4\varsigma_{\min}g_k}}{2\varsigma_{\min}} \leq \frac{h_k + \sqrt{\varsigma_{\min} g_k}}{\varsigma_{\min}},
	\end{align*}
	where in the last step we used $\sqrt{u+v} \leq \sqrt{u} + \sqrt{v}$ for any $u, v \geq 0$.
\end{proof}
\begin{lemma}\label{lem:calNsubsetcompact}
	Let $f \colon \calM \to \reals$ be twice continuously differentiable. 
	Let $\{ (x_0, s_0), (x_1, s_1), \ldots \}$ be the points and steps generated by Algorithm~\ref{algo:ARC}.
	Consider the following subset of $\calM$, obtained by collecting all curves generated by retracted steps (both accepted and rejected):
	\begin{align}
		\calN & = \bigcup_{k} \left\{ \Retr_{x_k}(ts_k) : t \in [0, 1] \right\}.
		\label{eq:calNretractedsteps}
	\end{align}
	If the sequence $\{x_0, x_1, x_2, \ldots\}$ remains in a compact subset of $\calM$, then $\calN$ is included in a compact subset of $\calM$.
\end{lemma}
\begin{proof}
	If $\calM$ is compact, the claim is clear since $\calN \subseteq \calM$. Otherwise, we use Lemma~\ref{lem:boundessteps}. Specifically, considering the upper bound in that lemma, define
	\begin{align*}
		\alpha(x) & = \sqrt{\frac{3 \|\nabla \hat f_x(0)\|}{\varsigma_{\min}}} + \frac{3}{2\varsigma_{\min}} \max\left(0, -\lambdamin(\nabla^2 \hat f_x(0))\right),
	\end{align*}
	where $\hat f_x = f \circ \Retr_x$. This is a continuous function of $x$, and $\|s_k\| \leq \alpha(x_k)$. Since by assumption $\{x_0, x_1, \ldots\} \subseteq \calK$ with $\calK$ compact, we find that
	\begin{align*}
		\forall k, \quad \|s_k\| \leq \sup_{k'} \alpha(x_{k'}) \leq \max_{x \in \calK} \alpha(x) \triangleq r,
	\end{align*}
	where $r$ is a finite number. Consider the following subset of the tangent bundle $\T\calM$:
	\begin{align*}
		\calK' & = \{ (x, s) \in \T\calM : x \in \calK, \|s\|_x \leq r \}.
	\end{align*}
	Since $\calK$ is compact, $\calK'$ is compact. 
	Furthermore, since the retraction is a continuous map, $\Retr(\calK')$ is compact, and it contains $\calN$.
\end{proof}
\begin{lemma} \label{lem:coreOfThmAboutA2}
	Let $f \colon \calM \to \reals$ be three times continuously differentiable, and consider the points and steps $\{(x_0, s_0), (x_1, s_1), \ldots \}$ generated by Algorithm~\ref{algo:ARC}.
	Assume the retraction is second-order nice on this set (see Definition~\ref{def:retrscndnice}).
	If the set $\calN$ as defined by~\eqref{eq:calNretractedsteps} is contained in a compact set $\calK$, then \aref{assu:firstorderregularityscalar} and~\aref{assu:firstorderregularityvectorretraction} are satisfied.
\end{lemma}
\begin{proof}
	
	
	For some $k$ and $\bar t \in [0, 1]$, let $(x, s) = (x_k, \bar t s_k)$ and define the pullback $\hat f = f \circ \Retr_x$. Notice in particular that $\Retr_x(s) \in \calN \subseteq \calK$.
	Combine the expression for the Hessian of the pullback~\eqref{eq:hessianpullback} with~\eqref{eq:lipschitzhessianbasic} to get:
	\begin{align*}
	\left\| \nabla^2 \hat f(s) - \nabla^2 \hat f(0) \right\|_{\mathrm{op}} & \leq \left\| T_{s}^* \circ \Hess f(\Retr_x(s)) \circ T_{s} - \Hess f(x) \right\|_{\mathrm{op}} + \left\|  W_{s} - W_0 \right\|_{\mathrm{op}}.
	\end{align*}
	By definition of $W_s$~\eqref{eq:Whessianpullback}, using the third condition on the retraction, we find that $W_0 = 0$ and
	\begin{align*}
	\|W_s\|_\mathrm{op} & = \max_{\substack{\dot s \in \T_x\calM \\ \|\dot s\| \leq 1}} \left| \inner{W_s[\dot s]}{\dot s} \right| \leq \|\grad f(\Retr_x(s))\| \cdot \max_{\substack{\dot s \in \T_x\calM \\ \|\dot s\| \leq 1}} \|c''(0)\| \leq c_3 G \|s\|,
	\end{align*}
	where $G = \max_{y \in \calK} \|\grad f(y)\|$ is finite by compactness of $\calK$ and continuity of the gradient norm.
	Thus, it remains to show that
	\begin{align*}
	\left\| T_s^* \circ \Hess f(\Retr_x(s)) \circ T_s - \Hess f(x) \right\|_{\mathrm{op}} & \leq c' \|s\|
	\end{align*}
	for some constant $c'$. For an arbitrary $\dot s \in \T_x\calM$, owing to differentiability properties of $f$, 
	\begin{align}
		\inner{\big[T_s^* \circ \Hess f(\Retr_x(s)) \circ T_s - \Hess f(x)\big][\dot s]}{\dot s} & = \int_{0}^{1} \frac{\mathrm{d}}{\mathrm{d}\tauort}  \inner{T_{\tauort s}^* \circ \Hess f(\Retr_x(\tauort s)) \circ T_{\tauort s} [\dot s]}{\dot s}  \mathrm{d}\tauort.
		\label{eq:integralddth}
	\end{align}
	We aim to upper bound the above by $c' \|s\| \|\dot s\|^2$.
	Consider the curve $c(\tauort) = \Retr_{x}(\tauort s)$ and a tangent vector field $U(\tauort) = T_{\tauort s} \dot s$ along $c$. Then, define
	\begin{align*}
	h(\tauort) & = \inner{T_{\tauort s}^* \circ \Hess f(c(\tauort)) \circ T_{\tauort s} [\dot s]}{\dot s} \\
	& = \inner{\Hess f(c(\tauort))[T_{\tauort s} \dot s]}{T_{\tauort s} \dot s} \\
	& = \inner{\Hess f(c(\tauort))[U(\tauort)]}{U(\tauort)}.
	\end{align*}
	The integrand in~\eqref{eq:integralddth} is the derivative of the real function $h$:
	\begin{align*}
	h'(\tauort) & = \frac{\mathrm{d}}{\mathrm{d}\tauort} \inner{\Hess f(c(\tauort))[U(\tauort)]}{U(\tauort)} \\
	& = \inner{\frac{\D}{\mathrm{d}\tauort}\Big[\Hess f(c(\tauort))[U(\tauort)]\Big]}{U(\tauort)} + \inner{\Hess f(c(\tauort))[U(\tauort)]}{\frac{\D}{\mathrm{d}\tauort} U(\tauort)} \\
	& = \inner{\left( \nabla_{c'(\tauort)} \Hess f \right)[U(\tauort)]}{U(\tauort)} + 2\inner{\Hess f(c(\tauort))[U(\tauort)]}{U'(\tauort)},
	\end{align*}
	where $U'(\tauort) \triangleq \frac{\D}{\mathrm{d}\tauort} U(\tauort)$ and we used that the Hessian is symmetric. Here, $\nabla_{c'(\tauort)} \Hess f$ is the Levi--Civita derivative of the Hessian tensor field at $c(\tauort)$ along $c'(\tauort)$---see~\cite[Def.~4.5.7, p102]{dC92}
	for the notion of derivative of a tensor field. For every $\tauort$, the latter is a symmetric linear operator on the tangent space at $c(\tauort)$. By Cauchy--Schwarz,
	\begin{align*}
	|h'(\tauort)| & \leq \|\nabla_{c'(\tauort)} \Hess f\|_{\mathrm{op}} \|U(\tauort)\|^2 + 2 \|\Hess f(c(\tauort))\|_{\mathrm{op}} \|U(\tauort)\| \|U'(\tauort)\|.
	\end{align*}
	By compactness of $\calK$ and continuity of the Hessian, we can define
	\begin{align*}
		H & = \max_{y \in \calK} \|\Hess f(y)\|_{\mathrm{op}}.
	\end{align*}
	By linearity of the connection $\nabla$, if $c'(\tauort) \neq 0$,
	\begin{align*}
	\nabla_{c'(\tauort)} \Hess f & = \|c'(\tauort)\| \cdot \nabla_{\frac{c'(\tauort)}{\|c'(\tauort)\|}} \Hess f.
	\end{align*}
	Furthermore, $c'(\tauort) = T_{\tauort s} s$ has norm bounded by the first assumption on the retraction: $\|c'(\tauort)\| \leq c_1 \|s\|$. Thus, in all cases, by compactness of $\calK$ and continuity of the function $v \to \nabla_v \Hess f$ on the tangent bundle $\T\calM$, there is a finite $J$ as follows: 
	\begin{align*}
		\|\nabla_{c'(\tauort)} \Hess f \|_{\mathrm{op}} & \leq c_1 \|s\| \cdot \overbrace{\max_{\substack{y \in \calK, v \in \T_y\calM \\ \|v\| \leq 1}} \| \nabla_v \Hess f \|_{\mathrm{op}}}^{J}.
	\end{align*}
	Of course, $\|U(\tauort)\| \leq c_1 \|\dot s\|$. Finally, we bound $\|U'(\tauort)\|$ using the second property of the retraction: $\|U'(\tauort)\| \leq c_2\|s\| \|\dot s\|$. Collecting what we learned about $|h'(t)|$ and injecting in~\eqref{eq:integralddth},
	\begin{align*}
		\left|\inner{\big[T_s^* \circ \Hess f(\Retr_x(s)) \circ T_s - \Hess f(x)\big][\dot s]}{\dot s}\right| & \leq \int_{0}^{1} |h'(\tauort)| \mathrm{d}\tauort \leq \left[ c_1^3 J + 2 c_1 c_2 H \right] \|s\| \|\dot s\|^2.
	\end{align*}
	Finally, it follows from Lemma~\ref{lem:lipschitzhessianbasic} that~\aref{assu:firstorderregularityscalar} and~\aref{assu:firstorderregularityvectorretraction} hold with $L = L' = c_3G + 2 c_1 c_2 H + c_1^3 J$ and $q \equiv 0$. We note in closing that the constants $G, H, J$ can be related to the Lipschitz properties of $f$, $\grad f$ and $\Hess f$, respectively.
\end{proof}
The theorem we wanted to prove now follows as a direct corollary.
\begin{proof}[Proof of Theorem~\ref{thm:regularitycompact}]
	For the main result, simply combine Lemmas~\ref{lem:calNsubsetcompact} and~\ref{lem:coreOfThmAboutA2}. To support the closing statement, it is sufficient to verify that Algorithm~\ref{algo:ARC} is a descent method
	owing to the step acceptance mechanism and the first part of condition~\eqref{eq:firstorderprogress}.
\end{proof}

\section{Proofs from Section~\ref{sec:Dretr}: differential of retraction} \label{sec:proofsDretr}

\subsection*{Stiefel manifold}

Proposition~\ref{prop:stiefelmastercorollary} regarding the Stiefel manifold is a corollary of the following statement.
\begin{lemma} \label{lemma:stiefelbound}
	For the Stiefel manifold $\calM = \St(n,p)$ with the Q-factor retraction $\Retr$, for all $X \in \calM$ and $S \in \T_X\calM$,
	\begin{align*}
	\sigmamin(\D\Retr_X(S)) \geq 1 - 3\|S\|_{\mathrm{F}} - \frac{1}{2}\|S\|_{\mathrm{F}}^2,
	\end{align*}
	where $\|\cdot\|_{\mathrm{F}}$ denotes the Frobenius norm. Moreover, for the special case $p = 1$ (the unit sphere in $\Rn$), the retraction reduces to $\Retr_x(s) = \frac{x+s}{\|x+s\|}$ and we have for all $x \in \calM, s \in \T_x\calM$:
	\begin{align*}
	\sigmamin(\D\Retr_x(s)) = \frac{1}{1 + \|s\|^2}.
	\end{align*}
\end{lemma}
\begin{proof}
Let $X \in \St(n, p)$ and $S \in \T_X \St(n, p) = \{ \dot X \in \Rnp : \dot X\transpose  X + X\transpose \dot X = 0 \}$ be fixed. Define $Q,R$ as the thin $QR$-decomposition of $X + S$, that is, $Q$ is an $n \times p$ matrix with orthonormal columns and $R$ is a $p \times p$ upper triangular matrix with positive diagonal entries such that $X+S = QR$: this decomposition exists and is unique since $X+S$ has full column rank, as shown below~\eqref{eq:QRfull}. By definition, we have that $\Retr_X(S) = Q$. 

For a matrix $M$, define $\mathrm{tril}(M)$ as the lower triangular portion of the matrix $M$, that is, $\mathrm{tril}(M)_{ij} = M_{ij}$ if $i \geq j$ and 0 otherwise. Further define $\rho_{\mathrm{skew}}(M)$ as  
\[\rho_{\mathrm{skew}}(M) \defeq \mathrm{tril}(M) - \mathrm{tril}(M)\transpose.\]
As derived in~\citep[Ex.~8.1.5]{AMS08} (see also the erratum for the reference) we have a formula for the directional derivative of the retraction along any $Z \in \T_X \St(n, p)$:
\begin{align}
	\D\Retr_X(S)[Z] & = Q\rho_{\mathrm{skew}}(Q\transpose Z R^{-1}) + (I-QQ\transpose) Z R^{-1}.
	\label{eq:DRetrQR}
\end{align}
We first confirm that $R$ is always invertible. To see this, note that $S$ being tangent at $X$ means $S\transpose X + X\transpose S = 0$ and therefore
\begin{align}
	R\transpose R & = \underbrace{(X+S)\transpose (X + S)}_{\textrm{start reading here}} = X\transpose X + X\transpose S + S\transpose X + S\transpose S = I_p + S\transpose S,
	\label{eq:QRfull}
\end{align}
which shows $R$ is invertible. Moreover the above expression also implies that:
\begin{align*}
	\sigma_k(R) = \sigma_k(X + S) = \sqrt{\lambda_k((X+S)\transpose (X+S))} = \sqrt{1 + \lambda_k(S\transpose S)} = \sqrt{1 + \sigma_k(S)^2},
\end{align*}
where $\sigma_k(M)$ represents the $k$th singular value of $M$ and $\lambda_k$ likewise extracts the $k$th eigenvalue (in decreasing order for symmetric matrices).
In particular we have that
\begin{align}
	\sigmamin(R^{-1}) & = \frac{1}{\sqrt{1 + \sigmamax(S)^2}} \geq \frac{1}{\sqrt{1 + \|S\|_{\mathrm{F}}^2}} \geq 1 - \frac{1}{2}\|S\|_{\mathrm{F}}^2, \nonumber\\
	\sigmamax(R^{-1}) & = \frac{1}{\sqrt{1 + \sigmamin(S)^2}} \leq 1. \label{eqn:sigmamaxbound}
\end{align}
Further note that since $QR = X + S$, we have that $Q = (X+S)R^{-1}$ and therefore
\begin{align*}
	Q\transpose Z R^{-1} & = (R^{-1})\transpose (X+S)\transpose Z R^{-1} \nonumber\\
	& = (R^{-1})\transpose X\transpose Z R^{-1} + (R^{-1})\transpose S\transpose Z R^{-1}.
\end{align*}
The first term above is always skew-symmetric since $Z$ is tangent at $X$, so that $X\transpose Z + Z\transpose X = 0$. Furthermore, for any skew-symmetric matrix $M$, $\rho_{\mathrm{skew}}(M) = M$. Therefore, using~\eqref{eq:DRetrQR},
\begin{align}
	\D\Retr_X(S)[Z] & = Q\rho_{\mathrm{skew}}(Q\transpose Z R^{-1}) + (I-QQ\transpose) Z R^{-1} \nonumber \\
	& = Q\left(\rho_{\mathrm{skew}}(Q\transpose Z R^{-1}) - Q\transpose Z R^{-1} \right) + Z R^{-1} \nonumber\\
	& = Q\left(\rho_{\mathrm{skew}}((R^{-1})\transpose S\transpose Z R^{-1}) - (R^{-1})\transpose S\transpose Z R^{-1} \right) + Z R^{-1}, \label{eq:DRetrequalsplit}
\end{align}
where in the last step we used $XR^{-1} - Q = -SR^{-1}$.
Further note that for any matrix $M$ of size $p \times p$,
\begin{align}
	\|Q(\rho_{\mathrm{skew}}(M) - M)\|_{\mathrm{F}} = \|\mathrm{tril}(M) - \mathrm{tril}(M)\transpose - M\|_{\mathrm{F}} \leq 3 \|M\|_{\mathrm{F}}.
\end{align}
Hence, we have that,
\begin{align}
	\| \D\Retr_X(S)[Z] \|_{\mathrm{F}} & \geq \|ZR^{-1}\|_{\mathrm{F}} - 3\| (R^{-1})\transpose S\transpose Z R^{-1} \|_{\mathrm{F}} \nonumber\\
						& \geq \|Z\|_{\mathrm{F}} \left( \sigmamin(R^{-1}) - 3 \sigmamax(R^{-1})^2 \sigmamax(S) \right),
\end{align}
where we have used $\|A\|_{\mathrm{F}}\sigmamin(B) \leq \|AB\|_{\mathrm{F}} \leq \|A\|_{\mathrm{F}} \sigmamax(B)$ multiple times. Using the bounds on the singular values of $R^{-1}$ (derived in \eqref{eqn:sigmamaxbound}) we get that
\[
	\| \D\Retr_X(S)[Z] \|_{\mathrm{F}} \geq \|Z\|_{\mathrm{F}} \left( 1 - \frac{1}{2}\|S\|_{\mathrm{F}}^2 - 3 \|S\|_{\mathrm{F}} \right). 
\]
Since this holds for all tangent vectors $Z$, we get that
\[ \sigmamin(\D\Retr_X(S)) \geq 1 - 3 \|S\|_{\mathrm{F}} - \frac{1}{2}\|S\|_{\mathrm{F}}^2.\]
To prove a better bound for the case of $p=1$ (the sphere), we improve the analysis of the expression derived in~\eqref{eq:DRetrequalsplit}. Note that for $p = 1$, the matrix inside the $\rho_{\mathrm{skew}}$ operator is a scalar, whose skew-symmetric part is necessarily zero. Also note that $Q$ is a single column matrix with value $\frac{x + s}{\|x + s\|}$ and $R = \|x + s\|$. Also, $X\transpose S X\transpose Z = 0$ since $S,Z$ are tangent. Therefore,
\begin{align*}
	\D\Retr_X(S)[Z] & = Z R^{-1} - Q(R^{-1})\transpose S\transpose Z R^{-1} \\
	 				& = \frac{1}{\|x+s\|} \left(z - \frac{s\transpose z}{1+\|s\|^2}(x+s)\right) \\
	 				& = \frac{1}{\|x+s\|} \left(z - \frac{s\transpose z}{1+\|s\|^2} s - \frac{s\transpose z}{1+\|s\|^2} x\right).
\end{align*}
Since $x$ is orthogonal to $s$ and $z$,
\begin{align*}
	\|\D\Retr_x(s)[z]\|^2 & = \frac{1}{1+\|s\|^2} \left( \left\|z - \frac{s\transpose z}{1+\|s\|^2} s \right\|^2 + \left( \frac{s\transpose z}{1+\|s\|^2}\right)^2 \right) \\
					   	    & = \frac{1}{1+\|s\|^2} \left( \|z\|^2 - 2\frac{(s\transpose z)^2}{1+\|s\|^2}  + \left( \frac{s\transpose z}{1+\|s\|^2}\right)^2 (1 + \|s\|^2) \right) \\
					   	    & = \frac{1}{1+\|s\|^2} \left( \|z\|^2 - \frac{(s\transpose z)^2}{1+\|s\|^2}\right) \\
					   	    & \geq \|z\|^2 \frac{1}{1+\|s\|^2} \left( 1 - \frac{\|s\|^2}{1+\|s\|^2} \right) \\
					   	    & = \|z\|^2 \frac{1}{(1+\|s\|^2)^2}.
\end{align*}
The worst-case scenario is achieved when $z$ and $s$ are aligned. Overall, we get
\[
	\|\D\Retr_x(s)[z]\| \geq \|z\| \frac{1}{1+\|s\|^2},
\]
which establishes the bound for the sphere. 
\end{proof}

\subsection*{Differential of exponential map for manifolds with bounded curvature}


Proposition~\ref{prop:mastercorollaryexponetial} regarding the differential of the exponential map on complete manifolds with bounded sectional curvature follows as a corollary of the following statement.
\begin{lemma} \label{lemma:exponentialbound}
	Assume all sectional curvatures of $\calM$, complete, are bounded above by $C$:
	\begin{itemize}
		\item[] If $C \leq 0$, then $\sigmamin(\D\Exp_x(s)) = 1$;
		\item[] If $C = \frac{1}{R^2} > 0$ and $\|s\| \leq \pi R$, then $1 \geq \sigmamin(\D\Exp_x(s)) \geq \frac{\sin(\|s\|/R)}{\|s\|/R}$.
	\end{itemize}
	As usual, we use the convention $\sin(t)/t = 1$ at $t = 0$.
\end{lemma}
\begin{proof}
	This results from a combination of few standard facts in Riemannian geometry:
	\begin{enumerate}
		\item \citep[Prop.~10.10]{lee2018riemannian} Given any two tangent vectors $s, \dot s \in \T_x\calM$, $J(t) = \D\Exp_x(ts)[t\dot s]$ is the unique Jacobi field along the geodesic $\gamma(t) = \Exp_x(ts)$ satisfying $J(0) = 0$ and $\Ddt J(0) = \dot s$.
		
		\item In particular, if $\dot s = \alpha s$ for some $\alpha \in \reals$ so that $\dot s$ and $s$ are parallel, then
		\begin{align*}
			J(t) & = \D\Exp_x(ts)][t\dot s] = \left. \ddq \Exp_x(ts + qt\dot s) \right|_{q=0} = \left. \ddq \gamma(t + q\alpha t) \right|_{q=0} = \alpha t \gamma'(t) = t P_{ts}(\dot s),
		\end{align*}
		using $\gamma'(t) = P_{ts}(s)$.
		It remains to understand the case where $\dot s$ is orthogonal to $s$.
		
		\item \citep[Prop.~10.12]{lee2018riemannian} If $\calM$ has constant sectional curvature $C$, $\|s\| = 1$ and $\inner{s}{\dot s} = 0$, the Jacobi field above is given by:
		\begin{align*}
			J(t) & = s_C(t) P_{ts}(\dot s),
		\end{align*}
		where $P_{ts}$ denotes parallel transport along $\gamma$ as in~\eqref{eq:Ps} and
		\begin{align*}
			s_C(t) & = \begin{cases}
				t & \textrm{ if } C = 0, \\
				R \sin(t/R) & \textrm{ if } C = \frac{1}{R^2} > 0, \textrm{ and} \\
				R \sinh(t/R) & \textrm{ if } C = -\frac{1}{R^2}.
			\end{cases}
		\end{align*}
		This can be reparameterized to allow for $\|s\| \neq 1$.
		Evaluating at $t = 1$ and using linearity in $\dot s$, we find for any $s, \dot s \in \T_x\calM$ that
		\begin{align}
			\D\Exp_x(s)[\dot s] & = P_s\!\left(\dot s_\parallel + \frac{s_C(\|s\|)}{\|s\|} \dot s_\perp \right),
		\end{align}
		where $\dot s_\perp$ is the part of $\dot s$ which is orthogonal to $s$ and $\dot s_\parallel$ is the part of $\dot s$ which is parallel to $s$---this corresponds to expression~\eqref{eq:DExpConstantCurvature}. By isometry of parallel transport, it is a simple exercise in linear algebra to deduce that
		\begin{align*}
			\sigmamin(\D\Exp_x(s)) & = \min\left( 1,  \frac{s_C(\|s\|)}{\|s\|} \right).
		\end{align*}
		
		\item \citep[Thm.~11.9(a)]{lee2018riemannian} Consider the case where $\dot s$ is orthogonal to $s$ of unit norm once again: the Jacobi field comparison theorem states that if the sectional curvatures of $\calM$ are upper-bounded by $C$, then $\|J(t)\|$ is at least as large as what it would be if $\calM$ had constant sectional curvature $C$---with the additional condition that $\|s\| \leq \pi R$ if $C = 1/R^2 > 0$. This leads to the conclusion through similar developments as above, using also~\citep[Prop.~10.7]{lee2018riemannian} to separate the components of $J(t)$ that are parallel or orthogonal to $\gamma'(t)$.
		\qedhere
	\end{enumerate}

\end{proof}

\subsection*{Extending to general retractions}

In order to prove Theorem~\ref{thm:retractionsgeneralmainnew}, we first introduce a result from topology. We follow~\cite{berge1963topological}, including the blanket assumption that all encountered topological spaces are Hausdorff (page 65 in that reference)---this is the case for us so long as the topology of $\calM$ itself is Hausdorff, which most authors require as part of the definition of a smooth manifold. Products of topological spaces are equipped with the product topology. Neighborhoods are open. A \emph{correspondence} $\Gamma \colon Y \to Z$ maps points in $Y$ to subsets of $Z$.
\begin{definition}[Upper semicontinuous (u.s.c.)\ mapping] \label{def:uscmapping}
	A correspondence $\Gamma \colon Y \rightarrow Z$ between two topological spaces $Y, Z$ is a \emph{u.s.c.\ mapping} if, for all $y$ in $Y$, $\Gamma(y)$ is a compact subset of $Z$ and, for any neighborhood $V$ of $\Gamma(y)$, there exists a neighborhood $U$ of $y$ such that, for all $u \in U$, $\Gamma(u) \subseteq V$. 
\end{definition}
\begin{theorem}[{\protect\citet[Thm.~VI.2, p116]{berge1963topological}}] \label{thm:bergemain}
	If $\phi$ is an upper semicontinuous, real-valued function in $Y \times Z$ and $\Gamma$ is a u.s.c.\ mapping of $Y$ into $Z$ (two topological spaces) such that $\Gamma(y)$ is nonempty for each $y$, then the real-valued function $M$ defined by
	\begin{align*}
	M(y) = \max_{z \in \Gamma(y)} \phi(y,z)
	\end{align*}
	is upper semicontinuous. (Under the assumptions, the maximum is indeed attained.)
\end{theorem}

We use the above theorem to establish our result. Manifolds (including tangent bundles) are equipped with the natural topology inherited from their smooth structure.

\begin{proof}[Proof of Theorem~\ref{thm:retractionsgeneralmainnew}]
	It is sufficient to show that the function
	\begin{align}
	t(r) & = \inf_{(x, s) \in \T\calM : x \in \calU, \|s\|_x \leq r} \sigmamin(\D\Retr_x(s))
	\end{align}
	is lower semicontinuous from $\reals^+ = \{r \in \reals : r \geq 0 \}$ to $\reals$, with respect to their usual topologies. Indeed, $t(0) = 1$ owing to the fact that $\D\Retr_x(0)$ is the identity map for all $x$, and $t$ being lower semicontinuous means that it cannot ``jump down''. Explicitly, lower semicontinuity at $r = 0$ implies that, for all $\delta > 0$, there exists $a > 0$ such that for all $r \leq a$ we have $t(r) \geq t(0) - \delta = 1 - \delta \triangleq b$. 
	
	To this end, consider the correspondence $\Gamma \colon \reals^+ \to \T\calM$ defined by
	\begin{align}
	\Gamma(r) & = \{ (x, s) \in \T\calM : x \in \calU \textrm{ and } \|s\|_x \leq r \}.
	\end{align}
	Further consider the function $\phi \colon \reals^+ \times \T\calM \to \reals$ defined by $\phi(r, (x, s)) = -\sigmamin(\D\Retr_x(s))$. Then, $t(r) = -M(r)$, where
	\begin{align}
	M(r) & = \sup_{(x, s) \in \Gamma(r)} \phi(r, (x, s)).
	\label{eq:Mrsup}
	\end{align}
	Thus, we must show $M$ is upper semicontinuous. By Theorem~\ref{thm:bergemain}, this is the case if
	\begin{enumerate}
		\item $\phi$ is upper semicontinuous,
		\item $\Gamma(r)$ is nonempty and compact for all $r \geq 0$, and
		\item For any $r \geq 0$ and any neighborhood $\calV$ of $\Gamma(r)$ in $\T\calM$, there exists a neighborhood $I$ of $r$ in $\reals^+$ such that, for all $r' \in I$, we have $\Gamma(r') \subseteq \calV$.
	\end{enumerate}
	The first condition holds a fortiori since $\phi$ is continuous, owing to smoothness of $\Retr \colon \T\calM \to \calM$. The second condition holds since $\calU$ is nonempty and compact.
	For the third condition, we show in Lemma~\ref{lem:neighborhoodballsTM} below that there exists a continuous function $\Delta \colon \calU \to \reals$ (continuous with respect to the subspace topology) such that $\{ (x, s) \in \T\calM : x \in \calU \textrm{ and } \|s\|_x \leq \Delta(x) \} \subseteq \calV$ and $\Delta(x) > r$ for all $x \in \calU$ (if $\calM$ is not connected, apply the lemma to each connected component which intersects with $\calU$). As a result, $\min_{x\in\calU} \Delta(x) = r + \varepsilon$ for some $\varepsilon > 0$ (using $\calU$ compact), and $\Gamma(r+\varepsilon)$ is included in $\calV$. We conclude that $I = [0, r+\varepsilon)$ is a suitable neighborhood of $r$ to verify the condition.
\end{proof}

We now state and prove the last piece of the puzzle, which applies above with $r(x)$ constant ($L = 0$). Although the context is quite different, the first part of the proof is inspired by that of the tubular neighborhood theorem in~\citep[Thm.~5.25]{lee2018riemannian}.

\begin{lemma} \label{lem:neighborhoodballsTM}
	Let $\calU$ be any subset of a connected Riemannian manifold $\calM$ and let $r \colon \calU \to \reals^+$ be $L$-Lipschitz continuous with respect to the Riemannian distance $\dist$ on $\calM$, that is, 
	\begin{align*}
	\forall x, x' \in \calU, \quad |r(x) - r(x')| \leq L \dist(x, x').
	\end{align*}
	Consider this subset of the tangent bundle:
	\begin{align*}
	\left\{ (x, s) \in \T\calM : x \in \calU \textrm{ and } \|s\|_x \leq r(x) \right\}.
	\end{align*}
	For any neighborhood $\calV$ of this set in $\T\calM$, there exists an $(L+1)$-Lipschitz continuous function $\Delta \colon \calU \to \reals^+$ such that $\Delta(x) > r(x)$ for all $x \in \calU$ and
	\begin{align*}
	\left\{ (x, s) \in \T\calM : x \in \calU \textrm{ and } \|s\|_x \leq \Delta(x) \right\} \subseteq \calV.
	\end{align*}
\end{lemma}
\begin{proof}
	Consider the following open subsets of the tangent bundle, defined for each $x \in \calM$ and $\delta \in \reals$:
	\begin{align*}
	V_\delta(x) & = \left\{ (x', s') \in \T\calM : \dist(x, x') < \delta - r(x) \textrm{ and } \|s'\|_{x'} < \delta \right\}.
	\end{align*}
	Referring to these sets, define the function $\Delta \colon \calU \to \reals$ as:
	\begin{align*}
	\Delta(x) & = \sup\left\{ \delta \in \reals : V_\delta(x) \subseteq \calV \right\}.
	\end{align*}
	This is well defined since $V_{r(x)}(x) = \emptyset$, so that $\Delta(x) \geq r(x)$ for all $x$.
	If $\Delta(x) = \infty$ for some $x$, then $\calV = \T\calM$ and the claim is clear (for example, redefine $\Delta(x) = r(x) + 1$ for all $x$). Thus, we assume $\Delta(x)$ finite for all $x$. The rest of the 	proof is in two parts.
	
	\textbf{Step 1: $\Delta$ is Lipschitz continuous.} Pick $x, x' \in \calU$, arbitrary. We must show
	\begin{align*}
	\Delta(x) - \Delta(x') \leq (L+1) \dist(x, x').
	\end{align*}
	Then, by reversing the roles of $x$ and $x'$, we get $|\Delta(x) - \Delta(x')| \leq (L+1) \dist(x, x')$, as desired. If $\Delta(x) \leq (L+1) \dist(x, x')$, the claim is clear since $\Delta(x') \geq 0$. Thus, we now assume $\Delta(x) > (L+1) \dist(x, x')$. Define $\delta = \Delta(x) - (L+1) \dist(x, x') > 0$. It is sufficient to show that $V_\delta(x') \subseteq \calV$, as this implies $\Delta(x') \geq \delta = \Delta(x) - (L+1) \dist(x, x')$, allowing us to conclude. To this end, we show the first inclusion in:
	\begin{align*}
	V_\delta(x') \subseteq V_{\Delta(x)}(x) \subseteq \calV.
	\end{align*}
	Consider an arbitrary $(x'', s'') \in V_\delta(x')$. This implies two things: first, $\|s''\|_{x''} < \delta \leq \Delta(x)$, and second:
	\begin{align*}
	\dist(x'', x) & \leq \dist(x'', x') + \dist(x', x) \\ & < \delta - r(x') + \dist(x', x) \\ & = \Delta(x) - r(x) + r(x) - r(x') - L \dist(x, x') \\ & \leq \Delta(x) - r(x),
	\end{align*}
	where in the last step we used $r(x) - r(x') \leq L\dist(x, x')$ since $r$ is $L$-Lipschitz continuous on $\calU$.
	As a result, $(x'', s'')$ is in $V_{\Delta(x)}(x)$, which concludes this part of the proof.

	\textbf{Step 2: $\Delta(x) > r(x)$ for all $x \in \calU$.}
	%
	Pick $x \in \calU$, arbitrary: $\calV$ is a neighborhood of
	\begin{align}
	\left\{ (x, s) \in \T\calM : \|s\|_x \leq r(x) \right\}.
	\label{eq:DeltalargerthanrmanifoldsetA}
	\end{align}
	The claim is that there exists $\varepsilon > 0$ such that
	\begin{align}
	\left\{ (x', s') \in \T\calM : \dist(x, x') \leq \varepsilon \textrm{ and } \|s'\|_{x'} \leq r(x) + \varepsilon \right\}
	\label{eq:DeltalargerthanrmanifoldsetB}
	\end{align}
	is included in $\calV$. Indeed, that would show that $\Delta(x) \geq r(x) + \varepsilon > r(x)$. To show this, we construct special coordinates on $\T\calM$ around $x$.
	
	The (inverse of the) exponential map at $x$ restricted to tangent vectors of norm strictly less than $\inj(x)$ (the injectivity radius at $x$) provides a diffeomorphism $\varphi$ from $\calW \subseteq \calM$ (the open geodesic ball of radius $\inj(x)$ around $x$) to $B(0, \inj(x))$: the open ball centered around the origin in the Euclidean space $\Rd$, where $d = \dim\calM$. Additionally, from the chart $(\calW, \varphi)$, we extract coordinate vector fields on $\calW$: a set of smooth vector fields $W_1, \ldots, W_d$ on $\calW$ such that, at each point in $\calW$, the corresponding tangent vectors form a basis for the tangent space. We further orthonormalize this local frame (see~\citep[Prop.~2.8]{lee2018riemannian}) into a new local frame, $E_1, \ldots, E_d$, so that for each $x' \in \calW$ we have that $E_1(x'), \ldots, E_d(x')$ form an orthonormal basis for $\T_{x'}\calM$ (with respect to the Riemannian metric at $x'$). Then, the map
	\begin{align*}
	\psi(x', s') & = \left( \varphi(x'), \zeta(x', s') \right) & & \textrm{ with } & \zeta(x', s') & = \left( \inner{E_1(x')}{s'}_{x'}, \ldots, \inner{E_d(x')}{s'}_{x'} \right)
	\end{align*}
	establishes a diffeomorphism between $\T\calW$ and $B(0, \inj(x)) \times \Rd$, with the following properties:
	\begin{enumerate}
		\item $\dist(x, x') = \|\varphi(x')\|$ (in particular, $\varphi(x) = 0$), and
		\item For any $s', v' \in \T_{x'}\calM$, it holds $\inner{s'}{v'}_{x'} = \inner{\zeta(x', s')}{\zeta(x', v')}$.
	\end{enumerate}
	(Here, $\inner{\cdot}{\cdot}$ and $\|\cdot\|$ denote the Euclidean inner product and norm in $\Rd$.)
	
	Expressed in these coordinates (that is, mapped through $\psi$), the set in~\eqref{eq:DeltalargerthanrmanifoldsetA} becomes:
	\begin{align*}
	D_0 & = \{ 0 \} \times \bar B(0, r(x)),
	\end{align*}
	where $\bar B(0, r(x))$ denotes the closed Euclidean ball of radius $r(x)$ around the origin in $\Rd$. Of course, $\calV \cap \T\calW$ maps to a neighborhood of $D_0$ in $\Rd \times \Rd$: call it $O$. Similarly, the set in~\eqref{eq:DeltalargerthanrmanifoldsetB} maps to:
	\begin{align*}
	D_\varepsilon & = \bar B(0, \varepsilon) \times \bar B(0, r(x) + \varepsilon).
	\end{align*}
	It remains to show that there exists $\varepsilon > 0$ such that $D_\varepsilon$ is included in $O$. 
	
	Use this distance on $\Rd \times \Rd$: $\dist((y, z), (y', z')) = \max(\|y - y'\|, \|z - z'\|)$. This distance is compatible with the usual topology. For each $(0, z)$ in $D_0$, there exists $\varepsilon_z > 0$ such that
	\begin{align*}
	C(z, \varepsilon_z) = \left\{ (y', z') \in \Rd \times \Rd : \|y'\| < \varepsilon_z \textrm{ and } \|z - z'\| < \varepsilon_z \right\}
	\end{align*}
	is included in $O$ (this is where we use the fact that $\calV$---hence $O$---is open). The collection of open sets $C(z, \varepsilon_z/2)$ forms an open cover of $D_0$. Since $D_0$ is compact, we may extract a finite subcover, that is, we select $z_1, \ldots, z_n$ such that the sets $C(z_i, \varepsilon_{z_i}/2)$ cover $D_0$. Now, define $\varepsilon = \min_{i = 1, \ldots, n} \varepsilon_{z_i}/2$ (necessarily positive), and consider any point $(y, z) \in D_\varepsilon$. We must show that $(y, z)$ is in $O$. To this end, let $\bar z$ denote the point in $\bar B(0, r(x))$ which is closest to $z$. Since $(0, \bar z)$ is in $D_0$, there exists $i$ such that $(0, \bar z)$ is in $C(z_i, \varepsilon_{z_i}/2)$. As a result,
	\begin{align*}
	\|z - z_i\| & \leq \|z - \bar z\| + \|\bar z - z_i\| < \varepsilon + \varepsilon_{z_i}/2 \leq \varepsilon_{z_i}.
	\end{align*}
	Likewise, $\|y\| \leq \varepsilon \leq \varepsilon_{z_i}/2 < \varepsilon_{z_i}$. Thus, we conclude that $(y, z)$ is in $C(z_i, \varepsilon_{z_i})$, which is included in $O$. This confirms $D_\varepsilon$ is in $O$, so that the set in~\eqref{eq:DeltalargerthanrmanifoldsetB} is in $\calV$ for some $\varepsilon > 0$.
\end{proof}

\end{document}